\documentclass[11pt, usenames, dvipsnames]{amsart}
\usepackage[utf8]{inputenc}
\usepackage[T2A]{fontenc}

\usepackage{ulem}
\usepackage{color}

\usepackage{fullpage}

\usepackage{amsmath}
\usepackage{amssymb}

\usepackage{cite}

\usepackage{tikz}
\usepackage{tikz-cd}
\usetikzlibrary{decorations.pathmorphing}

\usepackage[all]{xy}
\usepackage{mathrsfs}
\usepackage{enumerate}

\usepackage{hyperref}
\hypersetup{
    colorlinks=true,
    linkcolor=blue,
    filecolor=magenta,      
    urlcolor=cyan,
    pdftitle={Gr\"obner Bases for Coloured Operads},
    bookmarks=true,
    pdfpagemode=FullScreen,
}

\newcommand\ttt{\text{-}}

	\makeatletter
	
	\newcommand{\Rmnum}[1]{\expandafter\@slowromancap\romannumeral #1@}
	\makeatother

	\newcommand{\al}{\alpha} \newcommand{\be}{\beta}

	 \newcommand{\Sg}{\Sigma}

	  \newcommand{\cF}{{\mathcal F}}
	 
	\newcommand{\cI}{{\mathcal I}} 
	 
	\newcommand{\cM}{{\mathcal M}}  
	\newcommand{\cP}{{\mathcal P}} \newcommand{\cQ}{{\mathcal Q}}

	\newcommand{\frkg}{{\mathfrak g}} \newcommand{\frkh}{{\mathfrak h}}

	\newcommand{\xrar}[1]{\xrightarrow{#1}}

	\DeclareMathOperator{\const}{{const}}

	\DeclareMathOperator{\Der}{{Der}}

	\DeclareMathOperator{\gr}{{gr}}

	\DeclareMathOperator{\id}{{id}}

	\newcommand{\Fin}{\mathsf{Fin}}
	\newcommand{\Ord}{\mathsf{Ord}}
	\newcommand{\Vect}{\mathsf{Vect}}

	\newcommand\SymOp{\mathsf{SymOp}}
	\newcommand\NSymOp{\mathsf{NSymOp}}
	\newcommand\ShfOp{\mathsf{ShfOp}}

	\newcommand\ICom{\mathsf{ICom}}
	\newcommand\DCom{\mathsf{DCom}}
	\newcommand\As{\mathsf{As}}
	\newcommand\Com{\mathsf{Com}}
	\newcommand\Lie{\mathsf{Lie}}

	\newcommand\LieR{\mathsf{Lie\ttt Rinehart}}
	\newcommand\DerCom{\mathsf{DerCom}}

	\newcommand{\AffHS}{\mathsf{AffHS}}
	\newcommand{\SC}{\mathsf{SC}}
	\newcommand{\HSC}{H_0(\mathsf{SC}^{\textrm{vor}})}
	\newcommand{\LP}{\mathsf{LP}}
	\newcommand{\MLie}{\mathsf{MLie}}

	\newcommand{\limto}{{\displaystyle\lim_{\longrightarrow}}}
	\newcommand{\rightlim}{\mathop{\limto}}
	\newcommand{\leftlim}{\mathop{\displaystyle\lim_{\longleftarrow}}}
	\newcommand{\limfromn}{\leftlim\limits_{\raise3pt\hbox{$n$}}}
	\newcommand{\limton}{\rightlim\limits_{\raise3pt\hbox{$n$}}}
	\newcommand{\rightlimit}[1]{\mathop{\lim\limits_{\longrightarrow}}\limits%
	                    _{\raise3pt\hbox{$\scriptstyle #1$}}}
	\newcommand{\leftlimit}[1]{\mathop{\lim\limits_{\longleftarrow}}\limits%
	                    _{\raise3pt\hbox{$\scriptstyle #1$}}}

	\DeclareMathOperator{\st}{{st}}
	\DeclareMathOperator{\unc}{{unc}}
	\DeclareMathOperator{\lt}{{lt}}
	\DeclareMathOperator{\rd}{{rd}}

	\newtheorem{theo}{Theorem}[section] 
	
	\newtheorem{pr}[theo]{Proposition}

	\newtheorem{cor}[theo]{Corollary}
	\newtheorem{example}[theo]{Example}

	\theoremstyle{definition}
	
	\newtheorem{definition}[theo]{Definition}

	\theoremstyle{remark}
	
	\newtheorem{remark}[theo]{Remark}

\tikzset{ver/.style={circle, draw,inner sep=2pt}}

	\newcommand{\unTreeSD}[2]{
		\begin{tikzpicture} [scale=0.6]
		    \coordinate (rt) at (0,-1);
		    \node [ver] (z) {$#1$};
		    \node (t) at (0,1) {\tiny{1}};
		    \node at (1,0) {$#2$};
		        \draw (z) edge (rt);
		        \draw [densely dotted] (t) edge (z); 
		\end{tikzpicture}
	}

	\newcommand{\unTreeDS}[2]{
		\begin{tikzpicture} [scale=0.6]
		    \coordinate (rt) at (0,-1);
		    \node [ver] (z) {$#1$};
		    \node (t) at (0,1) {\tiny{1}};
		    \node at (1,0) {$#2$};
		        \draw [densely dotted] (z) edge (rt);
		        \draw (t) edge (z); 
		\end{tikzpicture}
	}

	\newcommand{\quadTreeSSS}[4]{
		\begin{tikzpicture} [scale=0.6]
		    \coordinate (rt) at (0,-1);
		    \node[ver] (z) {{$#1$}};
		    \node (l) at (-1,1) {\tiny{$#2$}} ;
		    \node (r) at (1,1) {\tiny{$#3$}} ;
		    \node at (2,0) {$#4$};
		        \draw (z) edge (rt); 
		        \draw (l) edge (z);
		        \draw (r) edge (z);
		\end{tikzpicture}
	}

	\newcommand{\quadTreeSDS}[4]{
		\begin{tikzpicture} [scale=0.6]
		    \coordinate (rt) at (0,-1);
		    \node [ver] (z) {{$#1$}};
		    \node (l) at (-1,1) {\tiny{$#2$}} ;
		    \node (r) at (1,1) {\tiny{$#3$}} ;
		    \node at (2,0) {$#4$};
		        \draw (z) edge (rt); 
		        \draw [densely dotted] (l) edge (z);
		        \draw (r) edge (z);
		\end{tikzpicture}
	}

	\newcommand{\quadTreeDDS}[4]{
		\begin{tikzpicture} [scale=0.6]
		    \coordinate (rt) at (0,-1);
		    \node [ver] (z) {{$#1$}};
		    \node (l) at (-1,1) {\tiny{$#2$}} ;
		    \node (r) at (1,1) {\tiny{$#3$}} ;
		    \node at (2,0) {$#4$};
		        \draw [densely dotted] (z) edge (rt); 
		        \draw [densely dotted] (l) edge (z);
		        \draw (r) edge (z);
		\end{tikzpicture}
	}

	\newcommand{\quadTreeSSD}[4]{
		\begin{tikzpicture} [scale=0.6]
		    \coordinate (rt) at (0,-1);
		    \node [ver] (z) {{$#1$}};
		    \node (l) at (-1,1) {\tiny{$#2$}} ;
		    \node (r) at (1,1) {\tiny{$#3$}} ;
		    \node at (2,0) {$#4$};
		        \draw (z) edge (rt); 
		        \draw (l) edge (z);
		        \draw [densely dotted] (r) edge (z);
		\end{tikzpicture}
	}

	\newcommand{\quadTreeDSD}[4]{
		\begin{tikzpicture} [scale=0.6]
		    \coordinate (rt) at (0,-1);
		    \node [ver] (z) {{$#1$}};
		    \node (l) at (-1,1) {\tiny{$#2$}} ;
		    \node (r) at (1,1) {\tiny{$#3$}} ;
		    \node at (2,0) {$#4$};
		        \draw [densely dotted] (z) edge (rt); 
		        \draw (l) edge (z);
		        \draw [densely dotted] (r) edge (z);
		\end{tikzpicture}
	}

	\newcommand{\quadTreeDDD}[4]{
		\begin{tikzpicture} [scale=0.6]
		    \coordinate (rt) at (0,-1);
		    \node [ver] (z) {{$#1$}};
		    \node (l) at (-1,1) {\tiny{$#2$}} ;
		    \node (r) at (1,1) {\tiny{$#3$}} ;
		    \node at (2,0) {$#4$};
		        \draw [densely dotted] (z) edge (rt); 
		        \draw [densely dotted] (l) edge (z);
		        \draw [densely dotted] (r) edge (z);
		\end{tikzpicture}
	}

	\newcommand{\cubLeftTreeSSSSS}[6]{
		\begin{tikzpicture} [scale=0.6]
		    \coordinate (rt) at (0,-1);
		    \node [ver] (z) {{$#1$}};
		    \node [ver] (l1v1) at (-1,1) {{$#2$}};
		    \node (l1v2) at (1,1) {\tiny{$#5$}};
		    \node (l2v1) at (-2,2) {\tiny{$#3$}};
		    \node (l2v2) at (0,2) {\tiny{$#4$}};
		    \node at (2,0) {$#6$};
		        \draw (z) edge (rt);
		        \draw (l1v1) edge (z);
		        \draw (l1v2) edge (z); 
			\draw (l2v1) edge (l1v1);
			\draw (l2v2) edge (l1v1);
		\end{tikzpicture}
	}

	\newcommand{\cubLeftTreeSSDSS}[6]{
		\begin{tikzpicture} [scale=0.6]
		    \coordinate (rt) at (0,-1);
		    \node [ver] (z) {{$#1$}};
		    \node [ver] (l1v1) at (-1,1) {{$#2$}};
		    \node (l1v2) at (1,1) {\tiny{$#5$}};
		    \node (l2v1) at (-2,2) {\tiny{$#3$}};
		    \node (l2v2) at (0,2) {\tiny{$#4$}};
		    \node at (2,0) {$#6$};
		        \draw (z) edge (rt);
		        \draw (l1v1) edge (z);
		        \draw [densely dotted] (l1v2) edge (z); 
			\draw (l2v1) edge (l1v1);
			\draw (l2v2) edge (l1v1);
		\end{tikzpicture}
	}

	\newcommand{\cubLeftTreeSSSSD}[6]{
		\begin{tikzpicture} [scale=0.6]
		    \coordinate (rt) at (0,-1);
		    \node [ver] (z) {{$#1$}};
		    \node [ver] (l1v1) at (-1,1) {{$#2$}};
		    \node (l1v2) at (1,1) {\tiny{$#5$}};
		    \node (l2v1) at (-2,2) {\tiny{$#3$}};
		    \node (l2v2) at (0,2) {\tiny{$#4$}};
		    \node at (2,0) {$#6$};
		        \draw (z) edge (rt);
		        \draw (l1v1) edge (z);
		        \draw (l1v2) edge (z); 
			\draw (l2v1) edge (l1v1);
			\draw [densely dotted] (l2v2) edge (l1v1);
		\end{tikzpicture}
	}

	\newcommand{\cubLeftTreeSSSDS}[6]{
		\begin{tikzpicture} [scale=0.6]
		    \coordinate (rt) at (0,-1);
		    \node [ver] (z) {{$#1$}};
		    \node [ver] (l1v1) at (-1,1) {{$#2$}};
		    \node (l1v2) at (1,1) {\tiny{$#5$}};
		    \node (l2v1) at (-2,2) {\tiny{$#3$}};
		    \node (l2v2) at (0,2) {\tiny{$#4$}};
		    \node at (2,0) {$#6$};
		        \draw (z) edge (rt);
		        \draw (l1v1) edge (z);
		        \draw (l1v2) edge (z); 
			\draw [densely dotted] (l2v1) edge (l1v1);
			\draw (l2v2) edge (l1v1);
		\end{tikzpicture}
	}

	\newcommand{\cubLeftTreeSSDSD}[6]{
		\begin{tikzpicture} [scale=0.6]
		    \coordinate (rt) at (0,-1);
		    \node [ver] (z) {{$#1$}};
		    \node [ver] (l1v1) at (-1,1) {{$#2$}};
		    \node (l1v2) at (1,1) {\tiny{$#5$}};
		    \node (l2v1) at (-2,2) {\tiny{$#3$}};
		    \node (l2v2) at (0,2) {\tiny{$#4$}};
		    \node at (2,0) {$#6$};
		        \draw (z) edge (rt);
		        \draw (l1v1) edge (z);
		        \draw [densely dotted] (l1v2) edge (z); 
			\draw (l2v1) edge (l1v1);
			\draw [densely dotted] (l2v2) edge (l1v1);
		\end{tikzpicture}
	}

	\newcommand{\cubLeftTreeSDSSD}[6]{
		\begin{tikzpicture} [scale=0.6]
		    \coordinate (rt) at (0,-1);
		    \node [ver] (z) {{$#1$}};
		    \node [ver] (l1v1) at (-1,1) {{$#2$}};
		    \node (l1v2) at (1,1) {\tiny{$#5$}};
		    \node (l2v1) at (-2,2) {\tiny{$#3$}};
		    \node (l2v2) at (0,2) {\tiny{$#4$}};
		    \node at (2,0) {$#6$};
		        \draw (z) edge (rt);
		        \draw [densely dotted] (l1v1) edge (z);
		        \draw (l1v2) edge (z); 
			\draw (l2v1) edge (l1v1);
			\draw [densely dotted] (l2v2) edge (l1v1);
		\end{tikzpicture}
	}

	\newcommand{\cubLeftTreeSDSDS}[6]{
		\begin{tikzpicture} [scale=0.6]
		    \coordinate (rt) at (0,-1);
		    \node [ver] (z) {{$#1$}};
		    \node [ver] (l1v1) at (-1,1) {{$#2$}};
		    \node (l1v2) at (1,1) {\tiny{$#5$}};
		    \node (l2v1) at (-2,2) {\tiny{$#3$}};
		    \node (l2v2) at (0,2) {\tiny{$#4$}};
		    \node at (2,0) {$#6$};
		        \draw (z) edge (rt);
		        \draw [densely dotted] (l1v1) edge (z);
		        \draw (l1v2) edge (z); 
			\draw [densely dotted] (l2v1) edge (l1v1);
			\draw (l2v2) edge (l1v1);
		\end{tikzpicture}
	}

	\newcommand{\cubLeftTreeSDSDD}[6]{
		\begin{tikzpicture} [scale=0.6]
		    \coordinate (rt) at (0,-1);
		    \node [ver] (z) {{$#1$}};
		    \node [ver] (l1v1) at (-1,1) {{$#2$}};
		    \node (l1v2) at (1,1) {\tiny{$#5$}};
		    \node (l2v1) at (-2,2) {\tiny{$#3$}};
		    \node (l2v2) at (0,2) {\tiny{$#4$}};
		    \node at (2,0) {$#6$};
		        \draw (z) edge (rt);
		        \draw [densely dotted] (l1v1) edge (z);
		        \draw (l1v2) edge (z); 
			\draw [densely dotted] (l2v1) edge (l1v1);
			\draw [densely dotted] (l2v2) edge (l1v1);
		\end{tikzpicture}
	}

	\newcommand{\cubLeftTreeDSDSS}[6]{
		\begin{tikzpicture} [scale=0.6]
		    \coordinate (rt) at (0,-1);
		    \node [ver] (z) {{$#1$}};
		    \node [ver] (l1v1) at (-1,1) {{$#2$}};
		    \node (l1v2) at (1,1) {\tiny{$#5$}};
		    \node (l2v1) at (-2,2) {\tiny{$#3$}};
		    \node (l2v2) at (0,2) {\tiny{$#4$}};
		    \node at (2,0) {$#6$};
		        \draw [densely dotted] (z) edge (rt);
		        \draw (l1v1) edge (z);
		        \draw [densely dotted] (l1v2) edge (z); 
			\draw (l2v1) edge (l1v1);
			\draw (l2v2) edge (l1v1);
		\end{tikzpicture}
	}

	\newcommand{\cubLeftTreeDSDSD}[6]{
		\begin{tikzpicture} [scale=0.6]
		    \coordinate (rt) at (0,-1);
		    \node [ver] (z) {{$#1$}};
		    \node [ver] (l1v1) at (-1,1) {{$#2$}};
		    \node (l1v2) at (1,1) {\tiny{$#5$}};
		    \node (l2v1) at (-2,2) {\tiny{$#3$}};
		    \node (l2v2) at (0,2) {\tiny{$#4$}};
		    \node at (2,0) {$#6$};
		        \draw [densely dotted] (z) edge (rt);
		        \draw (l1v1) edge (z);
		        \draw [densely dotted] (l1v2) edge (z); 
			\draw (l2v1) edge (l1v1);
			\draw [densely dotted] (l2v2) edge (l1v1);
		\end{tikzpicture}
	}

	\newcommand{\cubLeftTreeDDSDS}[6]{
		\begin{tikzpicture} [scale=0.6]
		    \coordinate (rt) at (0,-1);
		    \node [ver] (z) {{$#1$}};
		    \node [ver] (l1v1) at (-1,1) {{$#2$}};
		    \node (l1v2) at (1,1) {\tiny{$#5$}};
		    \node (l2v1) at (-2,2) {\tiny{$#3$}};
		    \node (l2v2) at (0,2) {\tiny{$#4$}};
		    \node at (2,0) {$#6$};
		        \draw [densely dotted] (z) edge (rt);
		        \draw [densely dotted] (l1v1) edge (z);
		        \draw (l1v2) edge (z); 
			\draw [densely dotted] (l2v1) edge (l1v1);
			\draw (l2v2) edge (l1v1);
		\end{tikzpicture}
	}

	\newcommand{\cubLeftTreeDDSSD}[6]{
		\begin{tikzpicture} [scale=0.6]
		    \coordinate (rt) at (0,-1);
		    \node [ver] (z) {{$#1$}};
		    \node [ver] (l1v1) at (-1,1) {{$#2$}};
		    \node (l1v2) at (1,1) {\tiny{$#5$}};
		    \node (l2v1) at (-2,2) {\tiny{$#3$}};
		    \node (l2v2) at (0,2) {\tiny{$#4$}};
		    \node at (2,0) {$#6$};
		        \draw [densely dotted] (z) edge (rt);
		        \draw [densely dotted] (l1v1) edge (z);
		        \draw (l1v2) edge (z); 
			\draw (l2v1) edge (l1v1);
			\draw [densely dotted] (l2v2) edge (l1v1);
		\end{tikzpicture}
	}

	\newcommand{\cubLeftTreeDDDSD}[6]{
		\begin{tikzpicture} [scale=0.6]
		    \coordinate (rt) at (0,-1);
		    \node [ver] (z) {{$#1$}};
		    \node [ver] (l1v1) at (-1,1) {{$#2$}};
		    \node (l1v2) at (1,1) {\tiny{$#5$}};
		    \node (l2v1) at (-2,2) {\tiny{$#3$}};
		    \node (l2v2) at (0,2) {\tiny{$#4$}};
		    \node at (2,0) {$#6$};
		        \draw [densely dotted] (z) edge (rt);
		        \draw [densely dotted] (l1v1) edge (z);
		        \draw [densely dotted] (l1v2) edge (z); 
			\draw (l2v1) edge (l1v1);
			\draw [densely dotted] (l2v2) edge (l1v1);
		\end{tikzpicture}
	}

	\newcommand{\cubLeftTreeDDDDS}[6]{
		\begin{tikzpicture} [scale=0.6]
		    \coordinate (rt) at (0,-1);
		    \node [ver] (z) {{$#1$}};
		    \node [ver] (l1v1) at (-1,1) {{$#2$}};
		    \node (l1v2) at (1,1) {\tiny{$#5$}};
		    \node (l2v1) at (-2,2) {\tiny{$#3$}};
		    \node (l2v2) at (0,2) {\tiny{$#4$}};
		    \node at (2,0) {$#6$};
		        \draw [densely dotted] (z) edge (rt);
		        \draw [densely dotted] (l1v1) edge (z);
		        \draw [densely dotted] (l1v2) edge (z); 
			\draw [densely dotted] (l2v1) edge (l1v1);
			\draw (l2v2) edge (l1v1);
		\end{tikzpicture}
	}

	\newcommand{\cubLeftTreeDDDDD}[6]{
		\begin{tikzpicture} [scale=0.6]
		    \coordinate (rt) at (0,-1);
		    \node [ver] (z) {{$#1$}};
		    \node [ver] (l1v1) at (-1,1) {{$#2$}};
		    \node (l1v2) at (1,1) {\tiny{$#5$}};
		    \node (l2v1) at (-2,2) {\tiny{$#3$}};
		    \node (l2v2) at (0,2) {\tiny{$#4$}};
		    \node at (2,0) {$#6$};
		        \draw [densely dotted] (z) edge (rt);
		        \draw [densely dotted] (l1v1) edge (z);
		        \draw [densely dotted] (l1v2) edge (z); 
			\draw [densely dotted] (l2v1) edge (l1v1);
			\draw [densely dotted] (l2v2) edge (l1v1);
		\end{tikzpicture}
	}

	\newcommand{\cubRightTreeSSSSS}[6]{
		\begin{tikzpicture} [scale=0.6]
		    \coordinate (rt) at (0,-1);
		    \node[ver] (z){$#1$};
		    \node (l1v1) at (-1,1) {\tiny{$#3$}};
		    \node[ver] (l1v2) at (1,1) {$#2$};
		    \node (l2v1) at (0,2) {\tiny{$#4$}};
		    \node (l2v2) at (2,2) {\tiny{$#5$}};
		    \node at (2,0) {$#6$};
		        \draw (z) edge (rt);
		        \draw (l1v1) edge (z);
		        \draw (l1v2) edge (z);
			\draw (l2v1) edge (l1v2);
			\draw (l2v2) edge (l1v2);
		\end{tikzpicture}
	}

	\newcommand{\cubRightTreeSSSSD}[6]{
		\begin{tikzpicture} [scale=0.6]
		    \coordinate (rt) at (0,-1);
		    \node[ver] (z){$#1$};
		    \node (l1v1) at (-1,1) {\tiny{$#3$}};
		    \node[ver] (l1v2) at (1,1) {$#2$};
		    \node (l2v1) at (0,2) {\tiny{$#4$}};
		    \node (l2v2) at (2,2) {\tiny{$#5$}};
		    \node at (2,0) {$#6$};
		        \draw (z) edge (rt);
		        \draw (l1v1) edge (z);
		        \draw (l1v2) edge (z);
			\draw (l2v1) edge (l1v2);
			\draw [densely dotted] (l2v2) edge (l1v2);
		\end{tikzpicture}
	}

	\newcommand{\cubRightTreeSSSDS}[6]{
		\begin{tikzpicture} [scale=0.6]
		    \coordinate (rt) at (0,-1);
		    \node[ver] (z){$#1$};
		    \node (l1v1) at (-1,1) {\tiny{$#3$}};
		    \node[ver] (l1v2) at (1,1) {$#2$};
		    \node (l2v1) at (0,2) {\tiny{$#4$}};
		    \node (l2v2) at (2,2) {\tiny{$#5$}};
		    \node at (2,0) {$#6$};
		        \draw (z) edge (rt);
		        \draw (l1v1) edge (z);
		        \draw (l1v2) edge (z);
			\draw [densely dotted] (l2v1) edge (l1v2);
			\draw (l2v2) edge (l1v2);
		\end{tikzpicture}
	}

	\newcommand{\cubRightTreeSSDSD}[6]{
		\begin{tikzpicture} [scale=0.6]
		    \coordinate (rt) at (0,-1);
		    \node[ver] (z){$#1$};
		    \node (l1v1) at (-1,1) {\tiny{$#3$}};
		    \node[ver] (l1v2) at (1,1) {$#2$};
		    \node (l2v1) at (0,2) {\tiny{$#4$}};
		    \node (l2v2) at (2,2) {\tiny{$#5$}};
		    \node at (2,0) {$#6$};
		        \draw (z) edge (rt);
		        \draw (l1v1) edge (z);
		        \draw [densely dotted] (l1v2) edge (z);
			\draw (l2v1) edge (l1v2);
			\draw [densely dotted] (l2v2) edge (l1v2);
		\end{tikzpicture}
	}

	\newcommand{\cubRightTreeSSDDS}[6]{
		\begin{tikzpicture} [scale=0.6]
		    \coordinate (rt) at (0,-1);
		    \node[ver] (z){$#1$};
		    \node (l1v1) at (-1,1) {\tiny{$#3$}};
		    \node[ver] (l1v2) at (1,1) {$#2$};
		    \node (l2v1) at (0,2) {\tiny{$#4$}};
		    \node (l2v2) at (2,2) {\tiny{$#5$}};
		    \node at (2,0) {$#6$};
		        \draw (z) edge (rt);
		        \draw (l1v1) edge (z);
		        \draw [densely dotted] (l1v2) edge (z);
			\draw [densely dotted] (l2v1) edge (l1v2);
			\draw (l2v2) edge (l1v2);
		\end{tikzpicture}
	}

	\newcommand{\cubRightTreeSSDDD}[6]{
		\begin{tikzpicture} [scale=0.6]
		    \coordinate (rt) at (0,-1);
		    \node[ver] (z){$#1$};
		    \node (l1v1) at (-1,1) {\tiny{$#3$}};
		    \node[ver] (l1v2) at (1,1) {$#2$};
		    \node (l2v1) at (0,2) {\tiny{$#4$}};
		    \node (l2v2) at (2,2) {\tiny{$#5$}};
		    \node at (2,0) {$#6$};
		        \draw (z) edge (rt);
		        \draw (l1v1) edge (z);
		        \draw [densely dotted] (l1v2) edge (z);
			\draw [densely dotted] (l2v1) edge (l1v2);
			\draw [densely dotted] (l2v2) edge (l1v2);
		\end{tikzpicture}
	}

	\newcommand{\cubRightTreeSDSSS}[6]{
		\begin{tikzpicture} [scale=0.6]
		    \coordinate (rt) at (0,-1);
		    \node[ver] (z){$#1$};
		    \node (l1v1) at (-1,1) {\tiny{$#3$}};
		    \node[ver] (l1v2) at (1,1) {$#2$};
		    \node (l2v1) at (0,2) {\tiny{$#4$}};
		    \node (l2v2) at (2,2) {\tiny{$#5$}};
		    \node at (2,0) {$#6$};
		        \draw (z) edge (rt);
		        \draw [densely dotted] (l1v1) edge (z);
		        \draw (l1v2) edge (z);
			\draw (l2v1) edge (l1v2);
			\draw (l2v2) edge (l1v2);
		\end{tikzpicture}
	}

	\newcommand{\cubRightTreeSDSDS}[6]{
		\begin{tikzpicture} [scale=0.6]
		    \coordinate (rt) at (0,-1);
		    \node[ver] (z){$#1$};
		    \node (l1v1) at (-1,1) {\tiny{$#3$}};
		    \node[ver] (l1v2) at (1,1) {$#2$};
		    \node (l2v1) at (0,2) {\tiny{$#4$}};
		    \node (l2v2) at (2,2) {\tiny{$#5$}};
		    \node at (2,0) {$#6$};
		        \draw (z) edge (rt);
		        \draw [densely dotted] (l1v1) edge (z);
		        \draw (l1v2) edge (z);
			\draw [densely dotted] (l2v1) edge (l1v2);
			\draw (l2v2) edge (l1v2);
		\end{tikzpicture}
	}

	\newcommand{\cubRightTreeDSDSD}[6]{
		\begin{tikzpicture} [scale=0.6]
		    \coordinate (rt) at (0,-1);
		    \node[ver] (z){$#1$};
		    \node (l1v1) at (-1,1) {\tiny{$#3$}};
		    \node[ver] (l1v2) at (1,1) {$#2$};
		    \node (l2v1) at (0,2) {\tiny{$#4$}};
		    \node (l2v2) at (2,2) {\tiny{$#5$}};
		    \node at (2,0) {$#6$};
		        \draw [densely dotted] (z) edge (rt);
		        \draw (l1v1) edge (z);
		        \draw [densely dotted] (l1v2) edge (z);
			\draw (l2v1) edge (l1v2);
			\draw [densely dotted] (l2v2) edge (l1v2);
		\end{tikzpicture}
	}

	\newcommand{\cubRightTreeDSDDS}[6]{
		\begin{tikzpicture} [scale=0.6]
		    \coordinate (rt) at (0,-1);
		    \node[ver] (z){$#1$};
		    \node (l1v1) at (-1,1) {\tiny{$#3$}};
		    \node[ver] (l1v2) at (1,1) {$#2$};
		    \node (l2v1) at (0,2) {\tiny{$#4$}};
		    \node (l2v2) at (2,2) {\tiny{$#5$}};
		    \node at (2,0) {$#6$};
		        \draw [densely dotted] (z) edge (rt);
		        \draw (l1v1) edge (z);
		        \draw [densely dotted] (l1v2) edge (z);
			\draw [densely dotted] (l2v1) edge (l1v2);
			\draw (l2v2) edge (l1v2);
		\end{tikzpicture}
	}

	\newcommand{\cubRightTreeDDSSS}[6]{
		\begin{tikzpicture} [scale=0.6]
		    \coordinate (rt) at (0,-1);
		    \node[ver] (z){$#1$};
		    \node (l1v1) at (-1,1) {\tiny{$#3$}};
		    \node[ver] (l1v2) at (1,1) {$#2$};
		    \node (l2v1) at (0,2) {\tiny{$#4$}};
		    \node (l2v2) at (2,2) {\tiny{$#5$}};
		    \node at (2,0) {$#6$};
		        \draw [densely dotted] (z) edge (rt);
		        \draw [densely dotted] (l1v1) edge (z);
		        \draw (l1v2) edge (z);
			\draw (l2v1) edge (l1v2);
			\draw (l2v2) edge (l1v2);
		\end{tikzpicture}
	}

	\newcommand{\cubRightTreeDSDDD}[6]{
		\begin{tikzpicture} [scale=0.6]
		    \coordinate (rt) at (0,-1);
		    \node[ver] (z){$#1$};
		    \node (l1v1) at (-1,1) {\tiny{$#3$}};
		    \node[ver] (l1v2) at (1,1) {$#2$};
		    \node (l2v1) at (0,2) {\tiny{$#4$}};
		    \node (l2v2) at (2,2) {\tiny{$#5$}};
		    \node at (2,0) {$#6$};
		        \draw [densely dotted] (z) edge (rt);
		        \draw(l1v1) edge (z);
		        \draw [densely dotted] (l1v2) edge (z);
			\draw [densely dotted]  (l2v1) edge (l1v2);
			\draw [densely dotted] (l2v2) edge (l1v2);
		\end{tikzpicture}
	}

	\newcommand{\cubRightTreeDDDSD}[6]{
		\begin{tikzpicture} [scale=0.6]
		    \coordinate (rt) at (0,-1);
		    \node[ver] (z){$#1$};
		    \node (l1v1) at (-1,1) {\tiny{$#3$}};
		    \node[ver] (l1v2) at (1,1) {$#2$};
		    \node (l2v1) at (0,2) {\tiny{$#4$}};
		    \node (l2v2) at (2,2) {\tiny{$#5$}};
		    \node at (2,0) {$#6$};
		        \draw [densely dotted] (z) edge (rt);
		        \draw [densely dotted] (l1v1) edge (z);
		        \draw [densely dotted] (l1v2) edge (z);
			\draw (l2v1) edge (l1v2);
			\draw [densely dotted] (l2v2) edge (l1v2);
		\end{tikzpicture}
	}

	\newcommand{\cubRightTreeDDDDS}[6]{
		\begin{tikzpicture} [scale=0.6]
		    \coordinate (rt) at (0,-1);
		    \node[ver] (z){$#1$};
		    \node (l1v1) at (-1,1) {\tiny{$#3$}};
		    \node[ver] (l1v2) at (1,1) {$#2$};
		    \node (l2v1) at (0,2) {\tiny{$#4$}};
		    \node (l2v2) at (2,2) {\tiny{$#5$}};
		    \node at (2,0) {$#6$};
		        \draw [densely dotted] (z) edge (rt);
		        \draw [densely dotted] (l1v1) edge (z);
		        \draw [densely dotted] (l1v2) edge (z);
			\draw [densely dotted] (l2v1) edge (l1v2);
			\draw (l2v2) edge (l1v2);
		\end{tikzpicture}
	}

	\newcommand{\cubRightTreeDDDDD}[6]{
		\begin{tikzpicture} [scale=0.6]
		    \coordinate (rt) at (0,-1);
		    \node[ver] (z){$#1$};
		    \node (l1v1) at (-1,1) {\tiny{$#3$}};
		    \node[ver] (l1v2) at (1,1) {$#2$};
		    \node (l2v1) at (0,2) {\tiny{$#4$}};
		    \node (l2v2) at (2,2) {\tiny{$#5$}};
		    \node at (2,0) {$#6$};
		        \draw [densely dotted] (z) edge (rt);
		        \draw [densely dotted] (l1v1) edge (z);
		        \draw [densely dotted] (l1v2) edge (z);
			\draw [densely dotted] (l2v1) edge (l1v2);
			\draw [densely dotted] (l2v2) edge (l1v2);
		\end{tikzpicture}
	}

	\newcommand{\quartLeftTreeSDSDDDD}[8]{
		\begin{tikzpicture} [scale=0.6]
		    \coordinate (rt) at (0,-1);
		    \node [ver] (z) {$#1$};
		    \node [ver] (l1v1) at (-1,1) {$#2$};
		    \node (l1v2) at (1,1) {\tiny{$#7$}};
		    \node [ver] (l2v1) at (-2,2) {$#3$};
		    \node (l2v2) at (0,2) {\tiny{$#6$}};
		    \node (l3v1) at (-3,3) {\tiny{$#4$}};
		    \node (l3v2) at (-1,3) {\tiny{$#5$}};
		    \node at (2,0) {$#8$};
		        \draw (z) edge (rt);
		        \draw [densely dotted] (l1v1) edge (z);
		        \draw (l1v2) edge (z); 
			\draw [densely dotted] (l2v1) edge (l1v1);
			\draw [densely dotted] (l2v2) edge (l1v1);
			\draw [densely dotted] (l3v1) edge (l2v1);
			\draw [densely dotted] (l3v2) edge (l2v1);
		\end{tikzpicture}
	}

	\newcommand{\JTreeDSDS}[5]{
		\begin{tikzpicture} [scale=0.6]
		    \coordinate (rt) at (0,-1);
		    \node[ver] (z) {$#1$};
		    \node (l1v1) at (-1,1) {\tiny{$#3$}} ;
		    \node[ver] (l1v2) at (1,1) {$#2$} ;
		    \node (l2v1) at (1,2) {\tiny{$#4$}} ;
		    \node at (2,0) {$#5$};
		        \draw [densely dotted]  (z) edge (rt); 
		        \draw (l1v1) edge (z);
		        \draw [densely dotted] (l1v2) edge (z);
			\draw (l2v1) edge (l1v2);
		\end{tikzpicture}
	}

	\newcommand{\JTreeDDSD}[5]{
		\begin{tikzpicture} [scale=0.6]
		    \coordinate (rt) at (0,-1);
		    \node[ver] (z) {$#1$};
		    \node (l1v1) at (-1,1) {\tiny{$#3$}} ;
		    \node[ver] (l1v2) at (1,1) {$#2$} ;
		    \node (l2v1) at (1,2) {\tiny{$#4$}} ;
		    \node at (2,0) {$#5$};
		        \draw [densely dotted]  (z) edge (rt); 
		        \draw [densely dotted] (l1v1) edge (z);
		        \draw (l1v2) edge (z);
			\draw [densely dotted] (l2v1) edge (l1v2);
		\end{tikzpicture}
	}

	\newcommand{\highYTreeSDDS}[5]{
		\begin{tikzpicture}[scale=0.6]
		    \coordinate (rt) at (0,-1);
		    \node[ver] (z) {$#1$};
		    \node[ver] (l1v1) at (0,1) {$#2$} ;
		    \node (l2v1) at (-1,2) {\tiny{$#3$}} ;
		    \node (l2v2) at (1,2) {\tiny{$#4$}} ;
		    \node at (2,0) {$#5$};
		        \draw (z) edge (rt); 
		        \draw [densely dotted] (l1v1) edge (z);
			\draw [densely dotted] (l2v1) edge (l1v1);
			\draw (l2v2) edge (l1v1);
		\end{tikzpicture}
	}

	\newcommand{\highYTreeSDSD}[5]{
		\begin{tikzpicture}[scale=0.6]
		    \coordinate (rt) at (0,-1);
		    \node[ver] (z) {$#1$};
		    \node[ver] (l1v1) at (0,1) {$#2$} ;
		    \node (l2v1) at (-1,2) {\tiny{$#3$}} ;
		    \node (l2v2) at (1,2) {\tiny{$#4$}} ;
		    \node at (2,0) {$#5$};
		        \draw (z) edge (rt); 
		        \draw [densely dotted] (l1v1) edge (z);
			\draw (l2v1) edge (l1v1);
			\draw [densely dotted] (l2v2) edge (l1v1);
		\end{tikzpicture}
	}

	\newcommand{\highYTreeDSSS}[5]{
		\begin{tikzpicture}[scale=0.6]
		    \coordinate (rt) at (0,-1);
		    \node[ver] (z) {$#1$};
		    \node[ver] (l1v1) at (0,1) {$#2$} ;
		    \node (l2v1) at (-1,2) {\tiny{$#3$}} ;
		    \node (l2v2) at (1,2) {\tiny{$#4$}} ;
		    \node at (2,0) {$#5$};
		        \draw [densely dotted] (z) edge (rt); 
			\draw (l2v1) edge (l1v1);
			\draw (l2v2) edge (l1v1);
		        \draw (l1v1) edge (z);
		\end{tikzpicture}
	}

\numberwithin{equation}{section}
\begin{document}

\title{Gr\"obner bases for coloured operads}

\author{Vladislav Kharitonov}
\email{vakharitonov\_2@edu.hse.ru}

\author{Anton Khoroshkin}
\address{	
	National Research University Higher School of Economics, 
	20 Myasnitskaya street, Moscow 101000, Russia  \& 	
	Institute for Theoretical and Experimental Physics,  25 Bolshaya Cheremushkinskaya, Moscow 117259, Russia; 	
}
\email{akhoroshkin@hse.ru}

\begin{abstract} 
	In this work we provide a definition of a coloured operad as a monoid in some monoidal category, and develop the machinery of Gr\"obner bases for coloured operads.
	Among the examples for which we show the existance of a quadratic Gr\"obner basis we consider the seminal Lie-Rinehart operad whose algebras are pairs (functions, vector fields).
\end{abstract}
\maketitle

%

\tableofcontents

\section*{Introduction}

 Gr\"obner bases and the related concepts proved to be an extremely powerful tool for exploring different properties of a wide range of algebraic objects. The list of objects for which this machinery has been developed includes Lie algebras~\cite{Shirshov}, commutative algebras~\cite{Buch}, associative algebras~\cite{Bergman, Bokut},  symmetric operads~\cite{KH} and nonsymmetric operads~\cite{DotsenkoVallette}. We are going to extend the Gr\"obner bases machinery discovered in~\cite{KH} to the case of coloured operads. The special case of coloured operads on $2$ colours called $1\ttt2$-coloured operads was already worked out in~\cite{KHor_Univ_envlop}, however, the general case has additional complexity thanks to the action of symmetric group. 
 
The notion of a coloured operad generalizes the notion of a classical operad, allowing operations to handle objects of different nature. Usually coloured operads are defined either through the type of algebras they give rise to, or in purely combinatorial terms (as in the book by Yau~\cite{Yau}). However, neither of these approaches provides the notions required to define a Gr\"obner basis for an operad. 

The key ingredient for defining a Gr\"obner basis for a type of algebraic objects is an ordering of the monomial basis of the free object compatible with the algebraic structure.
As in the case with classical symmetric operads, the symmetric coloured operads do not admit any ordering compatible with operadic compositions, so we cannot hope to develop the desired notions directly.

We start with the definition of a coloured operad introduced by van~der~Laan in~\cite{vanderLaan} and then, in the spirit of \cite{KH}, we introduce the notion of a \emph{shuffle coloured operad}. The free shuffle coloured operads has the canonical monomial basis which admits necessary orderings.
There exists a forgetful functor from the symmetric coloured operads to the shuffle coloured operads, which allows us to transfer the acquired information back to the symmetric operad.

The approach discovered in~\cite{KH} proved to be fruitful in the case of the classical operads, providing tools for concrete computations (see the book by Bremner and Dotsenko~\cite{BremnerDotsenko}) and enabling the algorithmic realisation --- a \texttt{Haskell} package \texttt{Operads}~\cite{HaskellOperads}.

Section~\ref{secExamples} is devoted to the description of a quadratic Gr\"obner basis in several natural operads on $2$ colours:
\begin{itemize}
	\item[\S\ref{sec::Icom}]  the operad $\ICom$ governing a pair-- a commutative associative algebra and an ideal in it;
	\item[\S\ref{sec::AffHS}] the operad $\AffHS$ of affine homogeneous spaces discovered by Merkulov in~\cite{Merkulov};
    \item[\S\ref{sec::SC}] The $0$-th cohomology of the Swiss Cheese operad and its Koszul dual operad of Leibniz pairs.
    \item[\S\ref{subSecLR}] The Lie-Rinehart operad and the operad $\DerCom$ governing pairs: a commutative algebra and a Lie algebra of its derivations.
\end{itemize}
As a by-product we (re)prove that all aforementioned operads are Koszul and we compute the corresponding generating series of dimensions of operations. Moreover, the structure of the symmetric group actions is also clear in all cases we consider.

All statements regarding reducibility of certain $S$-polynomials in this work result from computations performed on a computing with a Python script written for the purposes of this paper.
We provide a sample of the script's output for the operad $\LieR$ in the appendix and suggest different 
extra arguments showing the reducibility of $S$-polynomials for other examples.

\section*{Acknowledgement}
We would like to thank V.\,Dotsenko for useful comments on the first draft of the text.
The research of A.Kh. was carried out within the HSE University Basic Research Program
and funded (jointly) by the Russian Academic Excellence Project '5-100'. 
The results of  Section~\S\ref{secGBases}  have been obtained under support of the RSF grant No.19-11-00275. 

\section{Notation and main definitions}

We employ the definition of a coloured operad introduced by van~der~Laan in \cite{vanderLaan}, rather than the more recent definitions presented in the book by Yau \cite{Yau}, for the former has the merit of being a functorial one.
\subsection{Notation}
\mbox{}

$\Bbbk$ --- a field of characteristic 0. 

$\Vect$ --- the category of finite-dimensional vector spaces over $\Bbbk$.

$\Fin$ --- the category of finite sets with surjections as morphisms.

$\Ord$ --- the category of finite ordered sets with order-preserving surjections as morphisms.

$\mathbf{n}$ --- the set $\{1, \ldots, n\}$.

$\Sg_n$ --- symmetric group over a set of $n$ elements.

\subsection{Coloured sets}
Fix a finite set $I$, called the \emph{colouring set}. An $I$-\emph{coloured set} (or $I$-\emph{set}) $S$  is a finite set $S$ endowed with a map of sets $\chi:S \to I$ called the \emph{colouring of} $S$. 

Note that for a coloured set $S$, $\Sg_{|S|}$ acts on $S$ by precomposing the colouring of $S$ with a given permutation $\sigma$. That is $\sigma: (S,\chi) \mapsto (S,\chi\circ\sigma)$.

We denote by $\Fin_I$ the category of $I$-sets with surjections of underlying sets as morphisms, and by $\Ord_I$ -- the category of ordered $I$-sets with order-preserving surjections of underlying sets as morphisms.

Denote by $\const_c: S \to I$ the \emph{constant colouring of $S$ with the colour $c$}, that is a colouring with $\const_c(s) = c \quad \forall s\in S$.

Given a colouring $\chi_1$ of the set $\mathbf{n}$ and a colouring $\chi_2$ of the set $\mathbf{m}$, define a colouring $\chi_1 \circ_l \chi_2$ of the set $\mathbf{n+m-1}$ for any $l \le n$ as follows:

\[
	\chi_1 \circ_l \chi_2(k) =
	\begin{cases}
		\chi_1(k) & \text{ if } k<l; \\
		\chi_2(k-l+1) & \text{ if } l \le k < l+m; \\
		\chi_1(k-m) & \text{ if } k \ge l+m.
	\end{cases}
\]

Given a colouring set $I = \{c_1, \ldots, c_d\}$ and a weight vector $\boldsymbol{m} = (m_1, \ldots, m_d)$, we define \emph{the standard colouring} for $\boldsymbol{m}$ to be the colouring of the set with cardinality $\sum m_i$, which assigns the first colour to the first $m_1$ elements of the set, the second colour to the next $m_2$ elements of the set and so on. We denote this colouring by $\st_{\boldsymbol{m}}$.

\subsection{Classical definition of a coloured operad }
An \emph{$I$-coloured collection} $\cP$ is a collection of sets $\cP(n, \chi, c)$ indexed by all $n>0$, all colourings $\chi$ of $\mathbf{n}$, and all colours $c \in I$, endowed with a right $\Sg_n$-action on $\bigoplus_{\chi} \cP(n, \chi, c)$ such that $\sigma:\cP(n, \chi, c) \mapsto \cP(n, \chi\sigma, c)$ for any $\sigma \in \Sg_n$.

The colours $\chi(1), \ldots, \chi(n)$ are called the \emph{input colours of $\cP(n, \chi, c)$}, and the colour $c$ is called the \emph{output colour of $\cP(n, \chi, c)$}. 

\begin{definition}
	\label{classical def} 
	A \textbf{\emph{coloured operad}} is an $I$-coloured collection $\cP$ endowed with a set of morphisms called \emph{partial compositions}:
	\[
	\circ_l: \cP(n, \chi_1, c) \otimes \cP(m, \chi_2, \chi_1(l)) \longrightarrow \cP(n+m-1, \chi_1\circ_l \chi_2, c) \quad \text{for all} \quad l\le n,
	\]
	and a set of \emph{identity elements} $\id_c \in \cP(1,\const_c,c)$, satisfying the following conditions:
	\begin{itemize}
	    \item \emph{\textbf{Sequential composition axiom}}: 
	    \[
	    	(\lambda \circ_t \mu) \circ_{t-1+r} \nu = \lambda \circ_t (\mu \circ_r \nu),
	    \]
	    for all $t \le l, \: r\le m$ and $\lambda \in \cP(l, \chi_1, c), \: \mu \in \cP(m,\chi_2,\chi_1(t)), \: \nu \in \cP(n,\chi_3,\chi_2(r))$.
	    \item \emph{\textbf{Parallel composition axiom}}: 
	    $$
	    (\lambda \circ_r \mu) \circ_{s-1+m} = (\lambda \circ_s \nu) \circ_r
	    \mu
	    $$
	    for all $r<s\le l$ and $\lambda \in \cP(l, \chi_1, c), \: \mu \in \cP(m, \chi_2, \chi_1(r)), \: \nu \in \cP(n, \chi_3, \chi_1(s))$.
	    \item \emph{\textbf{Identity axiom}}:
	    \begin{align*}
		    \id_c \circ_1 \nu = \nu, \\
		    \mu \circ_s \id_{\chi_2(s)} = \mu
	    \end{align*}
	    
	    for all $s \le m$ and $\nu \in \cP(n,\chi_1,c), \: \mu \in \cP(m,\chi_2,d)$.
	\end{itemize}
\end{definition}

\begin{remark}
It is common to define a coloured operad of $k$ colours by specifying the sets $\cP(m_1,\ldots,m_k, c)$ of operations with $m_i$ arguments of $i^{th}$ colour and the output colour $c$, and partial compositions of these operations. Also one specifies the symmetries these operations have, so $\cP(m_1,\ldots,m_k)$ is a $\Sg_{m_1}\times \cdots \times \Sg_{m_k}$-module. 

To refactor this definition into the definition of the above form, let $n$ be equal to the sum of $m_i$. Set $\cP(n, \st_{\boldsymbol{m}}, c) := \cP(m_1,\ldots,m_k, c)$, and the $\Sigma_n$-representation $\bigoplus_{\chi}\cP(n,\chi,c)$ is isomorphic to the induced representation from the $\Sigma_{m_1}\times\dots\times\Sigma_{m_k}$-representation $\cP(m_1,\ldots,m_k, c)$.
In particular, the $I$-coloured collection $\cP$ is uniquely defined by the $\Sigma^{I}$-collection $\cup_{m_i,c}\cP(m_1,\ldots,m_k, c)$.
\end{remark}

\subsection{Functorial definition of a coloured operad}

Recall that $(\Vect, \otimes, \Bbbk)$, $(\Fin, \sqcup, \emptyset)$, and $(\Ord, \oplus, \emptyset)$ are monoidal categories, where $\oplus$ denotes  the ordered sum of sets. It is clear that $\Fin_I$ and $\Ord_I$ are also monoidal categories.

\begin{definition}
\mbox{}
	\begin{enumerate}
		    \item A \emph{\textbf{nonsymmetric coloured collection}} $\cP$ is a monoidal contravariant functor from the category $\Ord_I$ to the category  $\Vect$.
		    \item A \emph{\textbf{symmetric coloured collection}} $\cP$ is a monoidal contravariant functor from the category $\Fin_I$ to the category  $\Vect$.
	\end{enumerate}
\end{definition}

\begin{remark}
\mbox{}
	\begin{enumerate}
		\item The coloured sets $(\mathbf{n}, \chi)$ and their morphisms form a skeleton of both categories $\Ord_I$ and $\Fin_I$, so a collection $\cP$ is completely determined by its values on all morphisms of the form $(\mathbf{m},\chi_2) \twoheadrightarrow (\mathbf{n},\chi_1)$.
		\item The coherence condition for a coloured collection $\cP$ reads that:
		    \[
		    \cP ((\mathbf{m},\chi_2) \twoheadrightarrow (\mathbf{n},\chi_1)) = \bigotimes_{s \in \mathbf{n}} \cP ((f^{-1}(s),\chi_2|_{f^{-1}(s)}) \twoheadrightarrow (s,\const_{\chi_1(s)})).
		    \]
		    Note that there is exactly one arrow from a coloured set $(\mathbf{n},\chi_1)$  to the coloured set $(\mathbf{1},\const_c)$. The image of this arrow under $\cP$ is the space $\cP(n,\chi_1,c)$ from the classical definition.
	\end{enumerate}
\end{remark}

Now we proceed to define the operadic compositions, which in this setting amount to the monoidal structure on collections. 

\begin{definition}
\begin{enumerate}
	\item Let $\cP$ and $\cQ$ be two nonsymmetric coloured collections. Define their \emph{nonsymmetric composition} by the formula
	    \[
	    (\cP \circ \cQ)(n,\chi_1,c) := \bigoplus_{(\mathbf{m},\chi_2)} \cP(m,\chi_2,c) \otimes \left[ \bigoplus_{f: \mathbf{n} \twoheadrightarrow \mathbf{m}} \cQ(f^{-1}(1),\chi_1,\chi_2(1))\otimes \ldots \otimes\cQ(f^{-1}(m),\chi_1,\chi_2(m)) \right],
	    \]
	where the inner sum is taken over all non-decreasing surjections $f$.
	
	\item Let $\cP$ and $\cQ$ be two nonsymmetric coloured collections. Define their \emph{shuffle composition} by the formula
	    $$
	(\cP \circ_{sh} \cQ)(n,\chi_1,c) := \bigoplus_{(\mathbf{m},\chi_2)} \cP(m,\chi_2,c) \otimes \left[ \bigoplus_{f: \mathbf{n} \twoheadrightarrow \mathbf{m}} \cQ(f^{-1}(1),\chi_1,\chi_2(1))\otimes \ldots \otimes\cQ(f^{-1}(m),\chi_1,\chi_2(m)) \right],
	    $$
	where the inner sum is taken over all shuffling surjections $f$, that is surjections for which $\mathsf{min}f^{-1}(i) < \mathsf{min}f^{-1}(j)$ whenever $i<j$
	    
	\item Let $\cP$ and $\cQ$ be two symmetric coloured collections. Define their \emph{symmetric composition} by the formula
	    $$
	(\cP \circ \cQ)(n,\chi_1,c) := \bigoplus_{(\mathbf{m},\chi_2)} \cP(m,\chi_2,c) \otimes_{\Bbbk\Sg_m} \left[ \bigoplus_{f: \mathbf{n} \twoheadrightarrow \mathbf{m}} \cQ(f^{-1}(1),\chi_1,\chi_2(1))\otimes \ldots \otimes\cQ(f^{-1}(m),\chi_1,\chi_2(m)) \right],
	    $$
	where the inner sum is taken over all surjections $f$.
	
	\noindent In all the above formulae one restricts $\chi_1$ to the respective set if necessary.
	    
	\item Define the functor $\cI$ as follows:
	    $$
	    \cI(n,\chi_1,c) = 
	    \begin{cases}
	    \Bbbk \quad & \text{if} \quad n=1 \quad \text{and} \quad \chi_1(1) = c;\\
	    0 \quad & \text{otherwise.}
	    \end{cases}
	    $$
\end{enumerate}
\end{definition}

\begin{remark}
In the definition of symmetric composition the action of $\Sg_m$ on the RHS if defined by permuting the underlying coloured set for the left factor of the tensor product and by

\begin{tikzcd} \bigoplus \limits_{f: \mathbf{n} \twoheadrightarrow \mathbf{m}} \cQ(f^{-1}(1),\chi_1,\chi_2(1))\otimes \ldots \otimes \cQ(f^{-1}(m),\chi_1,\chi_2(m))
\arrow[d,mapsto]\\
\bigoplus \limits_{\sigma f: \mathbf{n} \twoheadrightarrow \mathbf{m}} \cQ(f^{-1}(\sigma^{-1}(1)),\sigma^{-1}\chi_1,\chi_2(\sigma^{-1}(1))\otimes \ldots \otimes \cQ(f^{-1}(\sigma^{-1}(1)),\sigma^{-1}\chi_1,\chi_2(\sigma^{-1}(m))
\end{tikzcd}

on the right factor, thus ensuring that the colouring of outputs matches the colouring of inputs.
\end{remark}

It is straightforward to check that:
\begin{pr}
Each of the compositions defined above, together with the functor $\cI$, endows the underlying category with a structure of a strict monoidal category.
\end{pr}

\begin{definition}
\begin{enumerate}
    \item A \emph{nonsymmetric coloured operad} is a monoid in the category of nonsymmetric coloured collections with the monoidal structure given by the nonsymmetric composition. We denote the category of nonsymmetric coloured operads by $\NSymOp_I$.
    \item A \emph{shuffle coloured operad} is a monoid in the category of nonsymmetric coloured collections with the monoidal structure given by the shuffle composition. We denote the category of shuffle coloured operads by $\ShfOp_I$.
    \item A \emph{symmetric coloured operad} is a monoid in the category of symmetric coloured collections with the monoidal structure given by the symmetric composition. We denote the category of symmetric coloured operads by $\SymOp_I$.
\end{enumerate}
\end{definition}

\begin{remark}
Given an operad $\cP$ thus defined, one can retrieve the coloured operad structure on $\cP$ in the sense of the definition \ref{classical def}. Firstly, by the merit of the unit morphism $\cI \to \cP$ one obtains the identity elements $\mathbf{id}_c$. Then the partial composition 
\[
\circ_l: \cP(n, \chi_1, c) \otimes \cP(m, \chi_2, \chi_1(l)) \longrightarrow \cP(n+m-1, \chi_1\circ_l \chi_2, c)
\]
is the component of $\cP \circ \cP$ lying in $$\cP(n, \chi_1, c)\otimes \left[ \bigoplus \cP(1, \chi_1 |_1, \chi_1(1)) \otimes \ldots \otimes \cP(m, \chi_2, \chi_1(l)) \otimes \ldots \otimes\cP(1, \chi_1 |_n, \chi_1(n))\right].$$
\end{remark}

\begin{pr}
Two definitions of a coloured operad are equivalent.
\end{pr}

\begin{proof}
To prove this proposition one may repeat the proof of \emph{Proposition 5.3.4} in \cite{Loday} verbatim, keeping track of input-output colouring compatibility.
\end{proof}

\begin{remark}
The combinatorial constructions presented further in this work stem from the third way of describing an operad, that is regarding an operad as an algebra over a monad of rooted trees. We do not provide this description here, limiting ourself to the aspects required for computation. This approach is explored more thoroughly in \cite{Loday} for the uncoloured case, and the coloured case is the same up to substituting coloured rooted trees for uncoloured rooted trees.
\end{remark}

\subsection{Free Coloured Operads}
The first step on the path to the Gr\"obner bases is the notion of a free object. The free coloured operad and an operadic ideal in it are defined analogously to the respective notions for uncoloured operads. Here we provide definitions in the framework of the classical definition \ref{classical def}, and in the next section we will introduce a combinatorial realisation of these notions.

\begin{definition}
The \emph{free $I$-coloured operad} $\cF(\Upsilon)$ generated by the $\Sigma^I$-collection of operations $\Upsilon:=\cup\Upsilon(m_1,\ldots,m_d;c)$ 
is the result of applying all possible operadic compositions to all pairs of elements from the induced $I$-coloured collection of $\Sigma_n$ representations $\oplus_{\chi}\Upsilon(n,\chi,c)$.

An operadic ideal in the free $I$-coloured operad $\cF$ is the result of repetitively applying all possible operadic composition to all pairs of the form $(\al, \be)$ and $(\be, \al)$, where $\be$ is already in the ideal and $\al$ is an arbitrary element of $\cF$.
\end{definition}
\begin{remark}
The  notation $\cF(a_1,\ldots,a_k| b_1,\ldots, b_m)$ of a presentation of an operad by  a given  set of operations $\{a_1, \ldots, a_k\}$ and relations $b_j$'s with known symmetries and the $I$-colouring of inputs/outputs has the following meaning.
First, with each generator $a_s\in \cF(m_1,\ldots,m_k,c)$ one has to assign a linear basis of the induced representation $\Bbbk[\Sigma_{m_1+\ldots+m_k}] a_s$ and similarly one has to choose the basis of the representations of symmetric group generated by defining relations $b_j$'s.

In particular the quantity of generators and relations is much more comparing to the symmetric case.
For example, with a generator $a\in  \cF(m_1,\ldots,m_k,c)$ which is symmetric in each colour one has to assign $\binom{(m_1+\ldots+m_k)!}{m_1! \ldots m_k!  }$ different generators that correspond to different colourings $\chi:\{1,\ldots,\sum m_i\} \rightarrow I$ with $|\chi^{-1}(j)| = m_j$.
\end{remark}

\begin{definition}
	An operadic ideal generated by the set $B = \{b_1, \ldots, b_k\}$  in a free coloured operad $\cF$ is the result of repetitively applying all possible operadic composition to all pairs of the form $(\al, \be)$ and $(\be, \al)$, where $\be$ is already in the ideal and $\al$ is an arbitrary element of $\cF$.
\end{definition}

Suppose we have a free coloured operad $\cF$ with the set of generating operations $G$. Consider a free (uncoloured) operad $\cF_{\unc}$ generated by the same set of operations $G$ disregarding the matching of the colours rule. Then the set of possible compositions of  $\cF_{\unc}$ includes the set of possible compositions of $\cF$, and the resulting operations are the same. This observation yields the following 

\begin{pr}
\label{inclusion into uncoloured}
A free coloured operad $\cF$ (as a monoid, disregarding the colour grading) is the factor of the corresponding uncoloured operad  $\cF_{\unc}$ by the operadic ideal generated by all colour-matching relations. In particular, we have an inclusion of monoids $\cF \hookrightarrow \cF_{\unc}$.
\end{pr}

In practice it is convenient to have an explicit combinatorial description for operads. To obtain such a description, we regard an operad as an algebra over the monad of rooted trees. We provide all combinatorial definitions needed for our purposes in the next section \S\ref{secComb}. For greater detail on this we refer to \cite{Loday}\S 5.6.

\subsection{Generating series}

	Suppose that the cardinality of the set of colours is equal to $d$, so we say $I=\{c_1,\ldots,c_d\}$. With each $I$-coloured symmetric collection $\cP$ we assign a collection of $d$ formal power series:

	\begin{multline}
	F^{i}_{\cP}(t_1,\ldots,t_d):=\sum_{m_1,\ldots,m_d\geq 0}\frac{\dim \cP(m_1+\ldots+m_d, \chi , c_i )}{m_1 ! \ldots m_d!} t_1^{m_1}\ldots t_d^{m_d},
	\\
	 \text{ with } |\chi^{-1}(c_j)| = m_j \text{ for } j=1,\ldots,d 
	\end{multline}

	The vector $F^{I}_\cP(t_1,\ldots,t_d):= (F^1_{\cP},\ldots, F^{d}_{\cP})$ is called the generating series of the $I$-coloured symmetric collection $\cP$.
	
\begin{pr}
The composition of generating series of $I$-coloured collections equals the generating series of the composition of collections: $F^{I}_{\cP\circ \cQ} = F^{I}_{\cP}\circ F^{I}_{\cQ}$
\end{pr}{}

\begin{proof}
Let the generating series of $\cP$ be in the variables $t_j$ and the generating series of $\cQ$ in the variables $s_k$. Choose an arbitrary colour $c_i$ and consider the monomial $\frac{\dim \cP(\Sigma m_i, \chi , c_i )}{m_1 ! \ldots m_d!} t_1^{m_1}\ldots t_d^{m_d}$ in $F^i_{\cP}$. Substituting $t_j$ for $F^j_{\cQ}(s_1,\ldots,s_d)$ we have:

\begin{multline}
\frac{\dim \cP(\Sigma m_i, \chi , c_i )}{m_1 ! \ldots m_d!} (F^1_{\cQ})^{m_1}\ldots (F^d_{\cQ})^{m_d} = 
\frac{\dim \cP(\Sigma m_i, \chi , c_i )}{m_1 ! \ldots m_d!}
\times
\\
\times
\big( 
\underbrace{
    \sum_{p^1_1,\ldots,p^1_d\geq 0}
    \frac{\dim {\cQ}(\Sigma p^1_i, \chi , c_1 )}{p^1_1 ! \ldots p^1_d!} s_1^{p^1_1}\ldots s_d^{p^1_d}
}_{m_1 \text{times}}
\cdot
\ldots
\cdot
\underbrace{
    \sum_{p^d_1,\ldots,p^d_d\geq 0}
    \frac{\dim {\cQ}(\Sigma p^d_i, \chi , c_d )}{p^d_1 ! \ldots p^d_d!} s_1^{p^d_1}\ldots s_d^{p^d_d}
}_{m_d \text{times}}
\big)
\end{multline}
Note that if $\chi = \sigma \st_{\boldsymbol{w}}$ for a weight vector $\boldsymbol{w}$ and a permutation $\sigma$, then 
\[
\dim {\cQ}(\Sigma p^d_i; \chi , c_d ) = \dim {\cQ}(\Sigma p^d_i; \st_{\boldsymbol{w}}, c_d ),
\]
so the coefficient of $s_1^{p^j_1}\ldots s_d^{p^j_d}$ in the respective sum is equal to $\dim {\cQ}(\Sigma p^d_i; \st_{\boldsymbol{w}}, c_d )$.
\\
Let's trace where some fixed monomial $s_1^{r_1}\ldots s_d^{r_d}$ appears in this expression. It comes from a choice of a summand in each of the inner sums, that is from a partition of each $s^{r_j}$ into $\Sigma m_i$ summands.

Defining such partition for all $s^{r_j}$'s is the same as defining a surjective map of the coloured set $\pmb{r}$ for $r = \Sigma r_j$ (the colouring is read from the exponents in the partition) onto some coloured set $M$ such that $M$ has $m_i$ elements of $i^{th}$ colour. Choosing any particular $M$ accounts for an ordering of the inner sums. There are $m_1 ! \ldots m_d!$ such sets, and we denote by $M_{st}$ the one that has the colouring $\st_{\pmb{m}}$.

Denote by $C_f$ the coefficient given by the map $f$, namely:

\[
	C_f = \prod_{k = 0}^{\Sigma m_i} \dim {\cQ}(\Sigma f^{-1}(e_k); \st_{\boldsymbol{w}_k}, c_k ),
\]

where $e_k$ denotes the $k^{th}$ element of $M$ and $\boldsymbol{w}_k$ denotes the weight vector corresponding to $f^{-1}(e_k)$.

So the coefficient of $s_1^{r_1}\ldots s_d^{r_d}$ is equal to:
\begin{multline}
	C(r_1, \ldots, r_d) =
	\sum_{m_1,\ldots,m_d\geq 0; \chi} \big[ 
	\frac{\dim \cP(\Sigma m_i; \chi , c_i )}{m_1 ! \ldots m_d!} 
	\cdot
	\sum_{f: \pmb{r} \to M} C_f
	\big]
	=
	\\
	=
	\sum_{m_1,\ldots,m_d\geq 0; \chi} \big[ 
	\dim \cP(\Sigma m_i; \chi , c_i )
	\cdot
	\sum_{f: \pmb{r} \to M_{st}} C_f
	\big]
\end{multline}
This coefficient accounts for all colourings with colouring vector $(r_1, \ldots, r_d)$, so for any such colouring $\chi$ we should multiply this coefficient by $\frac{1}{r_1!\ldots r_d!}$, and the result is exactly the coefficient corresponding to $(\cP \circ \cQ) (\Sigma r_j, \chi, c)$ in $F_{\cP\circ \cQ}$.
\end{proof}{}

Note that one can also consider the generating series of characters of the product of symmetric groups.
Namely let $F_{\cP}^{I}(t_{ij})$ be a collection of $d:=|I|$ formal power series on $d$ families of variables $\{t_{i1},t_{i2},t_{i3},\ldots\}$ with $1\leq i \leq d:= |I|$  that are symmetric in each collection of variables:
$$
F_{\cP}^{I}(t_{ij}):=\sum_{m_1,\ldots,m_k} \mathsf{char}(\cP(m_1,\ldots,m_k,c)).
$$
The composition of collections corresponds to the plethystic substitution of characters:
$$
F_{\cP\circ\cQ}^{I} = F_{\cP}^{I} \circ F_{\cQ}^{I}  
$$
Recall, that the plethystic composition written in the basis of Newton power sums $$p_k(x_1,x_2,\ldots):=x_1^{k} + x_2^{k}+\ldots$$ can be written in the following way
\begin{multline*}
p_d\circ F(\ldots, p_1(t_{1i},t_{2i}, t_{3i}), \ldots, p_m(t_{1i},t_{2i}, \ldots t_{3i}), \ldots) = 
\\ =
 F(\ldots, p_d(t_{1i},t_{2i}, t_{3i}), \ldots, p_{dm}(t_{1i},t_{2i}, \ldots t_{3i}), \ldots).  
\end{multline*}

\subsection{Forgetful functor} 
\label{FFuncDef}

The forgetful functor $\Ord_I \to \Fin_I$ disregarding the ordering of sets gives rise to a forgetful functor $\cF$ from the category of symmetric coloured collections to the category of nonsymmetric coloured collections which forgets the \mbox{$\Sg$-module} structure of the vector space. 
By the same considerations as in Prop. 3 from \cite{KH}, this functor commutes with the operadic compositions in the following sense: for two symmetric coloured collections $\cP$ and $\cQ$:
\[
	\cF(\cP \circ \cQ) = \cF(\cP) \circ_{sh} \cF(\cQ).
\]
So this is in fact a functor from the category of symmetric coloured operads to the category of shuffle coloured operads. We will explore this functor in greater detail in section \ref{FFuncComb}.

\section{Combinatorial description}
\label{secComb}
In this section we give combinatorial descriptions for the free operad of each type.
\subsection{Coloured trees}

\begin{definition}
A  \textbf{\emph{coloured rooted tree}} is a non-empty directed tree such that:
\begin{itemize}
    \item Every vertex has at least one incoming edge (its \emph{inputs}) and exactly one outgoing edge (its \emph{output}).
    \item Edges are allowed tobe connected with only one vertex, such (half)edges are called \emph{external}. 
    \item There is exactly one outgoing external edge, this edge is called the \emph{output} of the tree. The free endpoint of the output is called the \emph{root} of the tree.
    \item The free endpoints of the incoming external edges are called the \emph{leaves} of the tree. We suppose the tree to be decorated, meaning that the leaves of the tree are bijectively marked with the elements of the set $\mathbf{n}$ (called \emph{labels}) for some $n$.
    \item All edges of the tree are coloured with the set $I$.
\end{itemize}
\end{definition}

A coloured rooted tree with one vertex is called a \emph{corolla}. A coloured rooted tree with no vertices is called a \emph{degenerate tree}.

We picture the trees to be growing from the root upward, so following the direction of edges one goes down the tree.

A planar representation of a directed tree is equivalent to an ordering of inputs for each vertex of the tree. We compare two inputs of a vertex by comparing the minimal label reachable going through each input up the tree, the input with the lesser reachable label is lesser.

Now our goal is, given a coloured collection $\cP$, construct a realisation of the free operad $F(\cP)$ generated by $\cP$. In all three cases the realisation will be given in terms of coloured rooted trees and the grafting operation on them. From now on by a \emph{tree} we will mean a coloured rooted tree.

\subsection{Free nonsymmetric coloured operad}
Let $\cP$ be a nonsymmetric coloured collection. Fix a basis $\mathbf{B}$ of $\cP$. Now we assign a planar tree to each element of $\mathbf{B}$.

First, to each identity element $\id_c$ we assign a degenerate tree of the corresponding colour. Then to an element $p$ of $\mathbf{B}$ belonging to $P(n, \chi, c)$ we assign a corolla with $n$ leaves with labels increasing from left to right, and we colour the leaves' edges according to $\chi$. We mark the vertex of the corolla by $p$.

We define the partial composition $T_1 \circ_l T_2$ of two trees by grafting $T_2$ on the input of $T_1$ labelled with $l$, provided that this input and the output of $T_2$ have the same colour. Otherwise we set the composition to be zero. 

The basis of the free operad $F(\cP)$ consists of all trees obtained by grafting procedure starting from the set of corollas. By definition, this basis is closed under partial composition. We will refer to the elements of this basis as the \emph{tree monomials}.

\subsection{Free shuffle coloured operad}
Let $\cP$ be a nonsymmetric coloured collection. We construct the set of degenerate trees and corollas similarly to the previous case. 
We define the partial composition $T_1 \circ_{l,\sigma} T_2$ of two trees by grafting $T_2$ on the input of $T_1$ labelled with $l$, provided that this input and the output of $T_2$ have the same colour. Otherwise we set the composition to be zero.

We label the inputs of the resulting tree the same way as with nonsymmetrical composition, and after that we act by $\sigma$ on labels coming from $T_2$ and the labels coming from $T_1$ to the right of the grafting site.

Note that the trees resulting from this procedure satisfy the \emph{shuffle condition}:

For each inner vertex of the tree, the smallest descendants in each subtree growing from this vertex form an increasing sequence.

Such trees are called \emph{shuffle trees}.

The basis of the free operad $F(\cP)$ consists of all trees obtained by this grafting procedure starting from the set of corollas.

\subsection{Free symmetric coloured operad}
Let $\cP$ be a symmetric coloured collection. We construct the set of degenerate trees and corollas similarly to the previous cases, but now we render our trees as not equipped with planarization.

As in the previous cases, we define the partial composition $T_1 \circ_{l,\sigma} T_2$ of two trees by grafting $T_2$ on the input of $T_1$ labelled with $l$, provided that this input and the output of $T_2$ have the same colour, and otherwise set the composition to be zero. 

We label the inputs of the resulting tree the same way as with nonsymmetrical composition, and after that we act by $\sigma$ on all labels of our tree. 

Again, the basis of the free operad $F(\cP)$ consists of all trees obtained by this grafting procedure starting from the set of corollas.

\subsection{Gradings}

A tree in the basis of the free operad $\cF$ has three separate gradings:

\begin{enumerate}
    \item \emph{Arity degree} -- the number of leaves of the tree. The space of elements of arity degree $n$ is $\cF(n)$.
    \item \emph{Operation degree} -- the number of inner vertices of the tree.
    \item \emph{Colour degree} -- a vector $\boldsymbol{c}$ with $\boldsymbol{c}_i$ equal to the number of inputs coloured with $i$ minus the number of outputs coloured with $i$ (which is $0$ or $1$).
\end{enumerate}{}

Note that all three gradings are additive under operadic compositions.

\begin{definition}
An element of the free operad is said to be \emph{homogeneous} if it is a sum of basis elements with the same arity degree.
\end{definition}{}

\subsection{Forgetful functor}
\label{FFuncComb}
We now return to the forgetful functor $\cF: \SymOp_I \to \NSymOp_I $ defined in section \ref{FFuncDef}. As all our computations will involve transferring from a symmetric coloured operad to the corresponding shuffle coloured operad, we would like to provide a more concrete description of this functor.

In our setting, operads are usually defined through generators and relations. First we need to determine how $\cF$ acts on the set of generators of an operad.

Recall that each space of operations $\cP(n, \chi, c)$ of a symmetric operad is an $\Sigma_n$-module, so each generator $g$ comes with the orbit of $g$ under the action of the permutation group. The forgetful functor erases this action, so we need to introduce a new generator for each operation in the orbit of $g$. After this we need to choose a planarization of the resulting generators so that they are legitimate elements in the shuffle operad.

\begin{example} Suppose we are given two generators $\alpha$ and $r$ of arity 2, such $\Sigma_2$ acts trivially on $\alpha$ and non-trivially on $r$. Then we will need to introduce a new generator $l: l(x, y) = r(y, x)$:
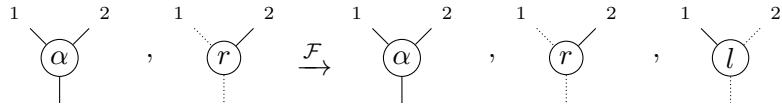
\[
	\quadTreeSSS{\al}{1}{2}{,}
	\quadTreeDDS{r}{1}{2}{\xrar{\cF}}
	\quadTreeSSS{\al}{1}{2}{,}
	\quadTreeDDS{r}{1}{2}{,}
	\quadTreeDSD{l}{1}{2}{}	
\]
\end{example}

Now we need to do the same with relations, which also come with their $\Sg$-orbits. For each relation and each permutation, act by the permutation on the relation, and then choose a planarization of the result such that all trees are shuffle trees.

\begin{example}
Suppose we have the following  quadratic relation on the generators from the previous example:
	\[
		\cubLeftTreeDDSDS{r}{r}{1}{3}{2}{-}
		\cubRightTreeDDSSS{r}{\al}{1}{2}{3}{}
	\]
The identity permutation will give us just the relation itself. The transposition $(12)$ will give us:
	\[
		\cubRightTreeDSDDS{l}{r}{1}{2}{3}{-}
		\cubLeftTreeDSDSS{l}{\al}{1}{3}{2}{}
	\]
The transposition $(23)$ yields:
	\[
		\cubLeftTreeDDSDS{r}{r}{1}{2}{3}{-}
		\cubRightTreeDDSSS{r}{\al}{1}{2}{3}{}
	\]
The transposition $(13)$ yields:
	\[
		\cubLeftTreeDDSSD{r}{l}{1}{3}{2}{-}
		\cubLeftTreeDSDSS{l}{\al}{1}{2}{3}{}
	\]
The cycle $c = (1 2 3)$ will give us:
	\[
		\cubLeftTreeDDSSD{r}{l}{1}{2}{3}{-}
		\cubLeftTreeDSDSS{l}{\al}{1}{3}{2}{}
	\]

And the cycle $c^2 = (1 3 2)$ will give us:

	\[
		\cubRightTreeDSDSD{l}{l}{1}{2}{3}{-}
		\cubLeftTreeDSDSS{l}{\al}{1}{2}{3}{}
	\]

\end{example}

\section{Gr\"obner bases}
\label{secGBases}

In this section we define all entities needed for the definition of a Gr\"obner basis.

\subsection{Admissible orderings}

From now on by \emph{operad} we mean \emph{shuffle coloured operad} unless specified otherwise. 

\begin{definition}
Let $\cF$ be a free operad. An ordering of the tree monomials of $\cF$ is said to be \emph{admissible} if the following holds:

\begin{enumerate}
    \item If $n < m$ then $\alpha < \beta$ for all $\alpha \in \cF(n), \beta \in \cF(m).$
    \item For $\alpha, \alpha' \in \cF(m)$, $\beta, \beta' \in \cF(n)$, if $\alpha \leq \alpha', \beta \leq \beta'$ then $\alpha \circ_{i, \omega} \beta \leq \alpha' \circ_{i, \omega} \beta'$ for all possible operadic compositions.
\end{enumerate}{}

\end{definition}{}

Our goal is to construct an admissible ordering of the monomials in the free operad. We claim that the construction of \emph{path-lexicographic ordering} from \cite{KH} can be transferred to the coloured setting. Recall that the path-lexicographic ordering for (non-coloured) shuffle operad is constructed as follows:

\begin{itemize}
    \item For a tree monomial  $\alpha \in \cF(n)$ construct a vector $\boldsymbol{a} = (a_1, \ldots, a_n)$, where $a_i$ is the word composed of vertex labels on the path from the root of the tree to the $i^{th}$ leaf, and a permutation $s \in S_n$ which is read from leaves left to right (recall that shuffle tree is planar).
    \item To compare two monomials, first compare their arities (the lengths of the sequence $(a_1, \ldots, a_n)$).
    \item If arities are equal, compare the vectors $\boldsymbol{a}$ component-wise using degree-lexicographic ordering on words.
    \item If vectors $\boldsymbol{a}$ are equal, compare the permutations using reverse lexicographic order. 
\end{itemize}{}

\begin{example}
For the tree monomial
	$
		\cubLeftTreeDDSSD{r}{l}{1}{3}{2}{}
	$
one has $(\boldsymbol{a}\,| \, s) = ((rl, r, rl) \, | \,(132))$.
\end{example}

\begin{remark}
Given the vector $\boldsymbol{a}$ and the colouring data of the generating operations one can restore the colourings of all edges of the tree, so this construction accounts for the colouring data as well.
\end{remark}

\begin{pr}
The path-lexicographic ordering is admissible.
\end{pr}

\begin{proof}
Recall the definition of $\cF_{\unc}$ from Proposition \ref{inclusion into uncoloured}. In \cite{KH} it is shown that the path-lexicographic ordering is admissible for $\cF_{\unc}$. But the requirements for being admissible in terms of coloured composition are less strict than for being admissible in terms of uncoloured composition, as the former is the subset of the latter. The trees of $\cF$ are also a subset of trees of $\cF_{\unc}$, and the restriction of an ordering on any subset is again an ordering.
\end{proof}

\subsection{QM-ordering}
\label{QMOrd}

The path-lexicographic ordering and its variations turn out to be inconvenient for calculating the Gr\"obner bases of some operads (e.g. the operad of Poisson algebras). In \cite{QMOrd} a new family of orderings was introduced, and we will employ an ordering of this type in our examples. We will call this type of orderings QM-orderings, which stands for Quantum Monomial.

The path-lexicographic ordering is based on the comparison of words in a free noncommutative algebra generated by the set of generators of the operad. The idea of a QM-ordering is to replace  monomials in the free noncommutative algebra by monomials in the algebra of quantum polynomials.

For our purposes it will suffice to construct QM-orderings for the operad with two generators $x, y$, so the algebra of quantum monomials is $\Bbbk \langle x, y \rangle / (xq - qx, yq - qy, yx - xyq)$ where $q$ is a formal parameter that commutes with $x$ and $y$. To compare two monomials in this algebra, first write them in the standard form $x^{a_1} y^{b_1}q^{c_1} <> x^{a_2}y^{b_2}q^{c_2}$. Then use the following rule
$$
x^{a_1} y^{b_1}q^{c_1} < x^{a_2}y^{b_2}q^{c_2} \Leftrightarrow 
\left[
\begin{array}{l}
a_1 > a_2, \\
(a_1 = a_2)\ \& \ (b_1 < b_2) \\
(a_1 = a_2)\ \& \ (b_1 = b_2)\ \& \ (c_1 < c_2)
\end{array}
\right.
$$

Having a comparison for words in the algebra, that is compatible with multiplication, we expand this ordering to an ordering on the free operad by associating a vector of words corresponding to paths from the root to leaves, same as we did for path-lexicographic ordering. We refer to \cite{QMOrd} for the proof that this is indeed an admissible ordering. In all computations were are dealing with  a $QM$-ordering the choice of the extension of the partial $QM$-ordering does not affect the story because the monomials that are not comparable with respect to a given $QM$-ordering do not interact with each other under the Buchberger algorithm.

\subsection{Divisibility}
Consider a tree monomial $\alpha$ with the underlying tree $T$. For a subtree $T'$ of $T$, containing all edges adjacent to the vertices of the subtree, we define a tree monomial $\alpha'$ as follows: the underlying coloured tree of $\alpha'$ is $T'$ and the labelling of the leaves is determined by the smallest descendant ordering, that is, the leaf with the smallest leaf label among its descendants gets the label $1$, the leaf with the same property among the yet unlabelled leaves gets the label $2$ and so on.

\begin{definition}
A tree monomial $\alpha$ \emph{is divisible by} a tree monomial $\beta$ if there is a subtree $T'$ of the underlying tree of $\alpha$, such that $\beta$ is the corresponding tree monomial for $T'$.
\end{definition}{}

As $\beta$ corresponds to a proper subtree of $\alpha$, we can obtain $\alpha$ by applying operadic compositions to $\beta$. This sequence of compositions can be applied to any tree monomial with the same number and colouring of the inputs and the output as $\beta$. This yields an operator on tree monomials which we denote by $m_{\alpha, \beta}$. Note that since the ordering of the tree monomials is compatible with the operadic compositions, if $\gamma < \beta$ then $m_{\alpha, \beta}(\gamma) < \alpha$.

\subsection{Reductions and $S$-polynomials}

In this section we recall the notions introduced in \cite{KH}, as they also suit the case of coloured operads.

\begin{definition}
For an element $f$ of the free operad its leading term $\lt(f)$ is the largest (in terms of the chosen admissible ordering) tree monomial in the expansion of $f$. The coefficient of $\lt(f)$ is called \emph{the leading coefficient} and denoted by $c_f$.
\end{definition}{}

\begin{definition}
For two homogeneous element $f$ and $g$ such that $\lt(f)$ is divisible by $\lt(g)$ we define \emph{reduction of $f$ modulo $g$} by the formula:
\[
 \rd_g(f)=f-\frac{c_f}{c_g}m_{\lt(f),\lt(g)}(g)
\]
 By construction we have $\lt(\rd_g(f)) < \lt(f)$.
\end{definition}{}

\begin{definition}
A tree monomial $\gamma$ is called a \emph{common multiple} of the tree monomials $\alpha$ and $\beta$, if it is divisible by both $\alpha$ and $\beta$. Tree monomials $\alpha$ and $\beta$ are said to have a \emph{small common 
multiple}, if they have a common multiple that is a union of two overlapping trees with one of these trees being isomorphic to $\alpha$ and another isomorphic to $\beta$ as a shuffle tree.
In particular, the  the number of vertices of the underlying tree is less than the total number of vertices for $\alpha$ and $\beta$. 
\end{definition}

Assume we have two homogeneous elements $f$ and $g$ whose leading terms have a small common multiple $\gamma$. In this setup we give the following definition:

\begin{definition}
\emph{The $S$-polynomial of $f$ and $g$ corresponding to $\gamma$} is defined by the formula:
\[
s_\gamma(f,g)=m_{\gamma,\lt(f)}(f)-\frac{c_f}{c_g} m_{\gamma,\lt(g)}(g).
\]
\end{definition}{}

\subsection{Gr\"obner bases}

\begin{definition}
Let $\cM$ be an operadic ideal  in a free $I$-coloured shuffle operad $\cF$ with a chosen admissible  ordering of monomials in $\cF^{I}$ and let $G$ be a set generating $\cM$. $G$ is called a
Gr\"obner basis of $\cM$ if for any element $f$ in $\cM$ the leading term of $f$ is divisible by the leading term of some element of $G$.
\end{definition}{}

This setting allows us to implement the classical Buchberger algorithm (for the description of the algorithm in operadic context we refer to Section \S3.7 of \cite{KH}.

Proposition~\ref{inclusion into uncoloured} says that the $I$-coloured shuffle operad $\cF^{I}/\cM$ is isomorphic to the quotient of the free shuffle (uncoloured) operad $\cF$ by the ideal $\widetilde{\cM}$ that is a union of $\cM$ and all compositions that contradicts the colouring. 

\begin{theo}
Let $G$ be a Gr\"obner basis of an operadic ideal  $\cM$  in a free $I$-coloured shuffle operad $\cF^{I}(a_1,\ldots,a_k)$ and let $B$ be the set $\{a_i\circ_l a_j\}$ of all partial composition of generators with the inconsistent colouring of the $l$'th input of $a_i$ and the output of $a_j$ considered as quadratic monomials in the free uncoloured shuffle operad $\cF(a_1,\ldots,a_n)$. Then the union $G\sqcup B$ constitutes the Gr\"obner basis of the ideal $\widetilde{\cM}\subset \cF$.
\end{theo}
\begin{proof}
Note that the set of colour mixing compositions constitute an ideal $\widetilde{B}\subset \cF$ generated by $B$. Therefore, each small common multiple $\gamma$ of the colour mixing relations $a_i\circ_l a_j$ and any element $\alpha\in G$ belongs to $\widetilde{B}$ and, in particular, the corresponding
 $S$-polynomial associated with $\gamma$ is reduced to zero using the relations from $B$. 
\end{proof}	

P. van der Laan explained in~\cite{vanderLaan} that $I$-coloured (co)operads admit Bar and coBar constructions and quadratic coloured operads admit the Koszul duality functor. One says that an $I$-coloured operad $\cP$ is Koszul whenever the coBar construction of its Koszul dual cooperad $\cP^{!}$ is quasi-isomorphic to $\cP$. In particular, the latter coBar construction $\Omega(\cP^{!})$ coincides with the minimal resolution of $\cP$ in the category of $I$-coloured operads.

\begin{theo}
\label{thm::PBW::Koszul}	
	Suppose that an $I$-coloured operad $\cP$ generated by the given set $\{a_1,\ldots,a_n\}$  admits a quadratic Gr\"obner basis with respect to an admissible ordering $\prec$ of monomials in the free $I$-coloured shuffle operad $\cF(a_1,\ldots,a_n)$.
	Then the $I$-coloured operad $\cP$ is Koszul and its coloured Koszul dual operad $\cP^{!}$ generated by the dual set of generators $\{a_1^{\vee},\ldots, a_n^{\vee}\}$ admits a quadratic Gr\"obner basis of relations with respect to the reverse admissible ordering of monomials $a^{\vee}\prec_{op}b^{\vee}\stackrel{\mathrm{def}}{\Leftrightarrow} a\succ b$ of the same arity/homogeneity in $\cF(a_1^{\vee},\ldots, a_n^{\vee})$.
\end{theo}
Note that the uncoloured shuffle operad associated with $\cP^{!}$ differs from the  shuffle operad that is Koszul dual to the uncoloured operad associated with $\cP$. Therefore, Theorem~\ref{thm::PBW::Koszul} does not follow from the analogous statement known for ordinary shuffle operads. However, the strategy of the proof is the same:
\begin{proof}
If $G$ is a linear basis of quadratic relations in the $I$-coloured operad $\cP$ and $\widehat{G}$ is the set of leading monomials of $G$ with respect to the partial ordering $\prec$ then the dual space $Ann(G)$ admits a linear basis $\bar{G}$ whose leading monomials with respect to the reverse ordering $\succ_{op}$ consists of the complement of $G$ in the set of quadratic monomials.
The associated graded $I$-coloured shuffle operad $\gr\cP$ has monomial quadratic relations $\widehat{G}$ and therefore is Koszul, its shuffle $I$-coloured Koszul dual operad $(\gr\cP)^{!}$ has also monomial relations that are indexed by the aforementioned complement of $G$.  What follows that the operad $\cP$ is Koszul and the dimensions of graded components of the coloured Koszul dual operad $\cP^{!}$ coincides with the dimensions of the corresponding graded components of $(\gr\cP)^{!}$. Consequently, $\bar{G}$ constitutes a Gr\"obner bases of $\cP^{!}$. 
\end{proof}

\section{Examples}
\label{secExamples}

\subsection{$\ICom$ operad}
\label{sec::Icom}

The $\ICom$ operad is a symmetric coloured operad on two colours generated by three operations:
	\begin{gather}
	\label{IComGen}
		\unTreeSD{i}{,}
		\quadTreeSSS{\al}{1}{2}{=}
		\quadTreeSSS{\al}{2}{1}{,}
		\quadTreeDDS{r}{1}{2}{}
	\end{gather}
subject to the following quadratic relations:
	\begin{gather}
\label{ICCom}
		\cubLeftTreeSSSSS{\al}{\al}{1}{2}{3}{=}
		\cubLeftTreeSSSSS{\al}{\al}{1}{3}{2}{=}
		\cubRightTreeSSSSS{\al}{\al}{1}{2}{3}{}
\end{gather}
\begin{gather}
\label{ICMult}
{		\cubLeftTreeDDSDS{r}{r}{1}{2}{3}{=}
		\cubLeftTreeDDSDS{r}{r}{1}{3}{2}{=}
		\cubRightTreeDDSSS{r}{\al}{1}{2}{3}{}
}	\end{gather}
	\begin{gather}
	\label{ICln1}
		\begin{tikzpicture} [scale=0.6]
		    \node[ver] (z) {$\alpha$};
		    \node[ver] (l1v1) at (-1,1) {$i$} ;
		    \node (l1v2) at (1,1) {\tiny{2}} ;
		    \node (l2v1) at (-1,2) {\tiny{1}} ;
		    \coordinate (rt) at (0,-1);
		    \node at (2,0) {$=$};
			\draw[densely dotted] (l2v1) edge (l1v1);
		        \draw (l1v1) edge (z);
		        \draw (l1v2) edge (z);
		        \draw  (z) edge (rt); 
		\end{tikzpicture}
		\begin{tikzpicture} [scale=0.6]
		    \node[ver] (z) {$i$};
		    \node[ver] (l1v1) at (0,1) {$r$} ;
		    \node (l2v1) at (-1,2) {\tiny{1}} ;
		    \node (l2v2) at (1,2) {\tiny{2}} ;
		    \coordinate (rt) at (0,-1);
		    \node at (2,0) {};
			\draw[densely dotted] (l2v1) edge (l1v1);
			\draw (l2v2) edge (l1v1);
		        \draw[densely dotted] (l1v1) edge (z);
		        \draw (z) edge (rt); 
		\end{tikzpicture}
\\
	\label{ICln2}
		\JTreeDDSD{r}{i}{1}{2}{=}
		\JTreeDDSD{r}{i}{2}{1}{}	
	\end{gather}
Relation~\eqref{ICCom} means that $\alpha$ is a commutative associative multiplication, Relation~\eqref{ICMult} says that $r$ defines an action of this commutative algebra and Relation~\eqref{ICln1} says that $i$ is a map of modules of this algebra.
A typical algebra over the operad $\ICom$ is a pair $(A, I)$ of a commutative algebra $A$ and an ideal $I \hookrightarrow A$. $\alpha$ corresponds to the multiplication in $A$, $r$ to the multiplication of an element of the ideal by an element of $A$, and $i$ corresponds to the inclusion of $I$ into $A$.

The corresponding coloured shuffle operad has four generating operations, we denote them by $i, \alpha, r$ and $l$:

	\begin{gather}
		\label{ICShGen}
		\unTreeSD{i}{,}
		\quadTreeSSS{\al}{1}{2}{,}
		\quadTreeDDS{r}{1}{2}{,}
		\quadTreeDSD{l}{1}{2}{}
	\end{gather}

\begin{theo}
\label{thm::ICom}	
The $\Sigma_3$ and $\Sigma_2$ orbits of the defining quadratic relations \eqref{ICCom}-\eqref{ICln2}   constitute a quadratic Gr\"obner basis of the ideal of relations of the $2$-coloured operad $\ICom$
if one considers the path lexicographic ordering of the monomials associated with the following ordering of generators:
$\alpha < i < l < r$. 

The generating series of dimensions of $\ICom$ are equal to 
$$
\overrightarrow{F_{\ICom}}(t_1,t_2):=(F^1_{\ICom}(t_1, t_2);F^1_{\ICom}(t_1, t_2))  = \left(e^{t_1 + t_2} - 1; e^{t_1 + t_2} - e^{t_1}\right)
$$
\end{theo}
First, let us act by the symmetric group on each of the relations to obtain the relations in the shuffle operad, and then find the leading term in each acquired relation. For each relation in the symmetric coloured operad $\ICom$ we list the leading terms of relations in the shuffle operad produced by it.

Relation \eqref{ICCom} yields:
	\begin{gather}
		\label{ICComLT}
		\cubLeftTreeSSSSS{\al}{\al}{1}{2}{3}{,}
		\cubLeftTreeSSSSS{\al}{\al}{1}{3}{2}{}
	\end{gather}

Relation \eqref{ICMult} yields:
	\begin{gather}
	\label{ICMultLT1}
		\cubLeftTreeDDSDS{r}{r}{1}{2}{3}{,}
		\cubLeftTreeDDSDS{r}{r}{1}{3}{2}{,}
		\cubLeftTreeDDSSD{r}{l}{1}{2}{3}{,}
		\cubLeftTreeDDSSD{r}{l}{1}{3}{2}{}
	\end{gather}
and:
	\begin{gather}
	\label{lCMultLT2} 
		\cubLeftTreeDSDSS{l}{\al}{1}{2}{3}{,}
		\cubLeftTreeDSDSS{l}{\al}{1}{3}{2}{}
	\end{gather}
Relation \eqref{ICln1} yields:
	\begin{gather}
	\label{ICIn1LT}
		\highYTreeSDDS{i}{r}{1}{2}{,}
		\highYTreeSDSD{i}{l}{1}{2}{}
	\end{gather}
And the relation \eqref{ICln2} yields:
	\begin{gather}
	\label{ICIn2LT}
		\begin{tikzpicture}[scale=0.6]
		    \node[ver] (z) {$l$};
		    \node[ver] (l1v1) at (-1,1) {$i$} ;
		    \node[minimum size=1pt] (l1v2) at (1,1) {\tiny{2}} ;
		    \node[minimum size=1pt] (l2v1) at (-1,2) {\tiny{1}} ;
		    \coordinate (rt) at (0,-1);
		    \node at (2,0) {};
			\draw [densely dotted] (l2v1) edge (l1v1);
		        \draw (l1v1) edge (z);
		        \draw [densely dotted] (l1v2) edge (z);
		        \draw [densely dotted]  (z) edge (rt); 
		\end{tikzpicture}
	\end{gather}

Define $\gr \ICom$ as the factor of the free operad generated by $i, \alpha, r, l$ by the operadic ideal spanned by all the leading terms listed above.  We claim that

\begin{pr}
Starting with arity $2$, all trees in $\gr \ICom$ have the following general form: the tree grows only to the right; from the root up, first come $N_l \ge 0$ vertices of type $l$; then either tree terminates or comes exactly one vertex of type $r$; then come $N_{\alpha} \ge 0$ vertices of type $\alpha$. Additionally, any of the $\alpha$-type vertices and the $r$-type vertex may have vertices of type $i$ grafted upon them (thus $i$-type vertices are always leaves):
	\begin{gather}
	\label{ICBigTree}
		\begin{tikzpicture}[scale=0.6]
		    \node[ver] (z) {$l$};
		    \node  (l1v1) at (-1,1) {\tiny{$1$}} ;
		    \node[ver] (l1v2) at (1,1) {$l$} ;
		    \node  (l2v1) at (0,2) {\tiny{$2$}} ;
		    \node   (l2v2) at (2,2) {$\cdots$} ;
		    \node[ver] (l3v1) at (3,3) {$l$} ;
		    \node  (l4v1) at (2,4) {\tiny{$N_l$}} ;
		    \node[ver] (l4v2) at (4,4) {$r$} ;
		    \node  (l5v1) at (3,5) {\tiny{$N_l + 1$}} ;
		    \node[ver] (l5v2) at (5,5) {$\alpha$} ;
		    \node[ver] (l6v1) at (4,6) {$i$} ;
		    \node   (l6v2) at (6,6) {$\cdots$} ;
		    \node  (l7v1) at (4,7) {\tiny{$N_l + 2$}} ;
		    \node[ver]  (l7v2) at (7,7) {$\alpha$} ;
		    \node[ver] (l8v1) at (6,8) {$i$} ;
		    \node   (l8v2) at (8,8) {\tiny{$N_l + N_{\alpha} + 2$}} ;
		    \node   (l9v1) at (6,9) {\tiny{$N_l + N_{\alpha} + 1$}} ;
		    \coordinate (rt) at (0,-1);
			\draw  [densely dotted] (l9v1) edge (l8v1);
		        \draw  (l8v1) edge (l7v2);
		        \draw (l8v2) edge (l7v2);
		        \draw [densely dotted]  (l7v1) edge (l6v1);
		        \draw (l7v2) edge (l6v2);
		        \draw  (l6v1) edge (l5v2);
		        \draw (l6v2) edge (l5v2);
		        \draw [densely dotted]  (l5v1) edge (l4v2);
		        \draw (l5v2) edge (l4v2);
		        \draw (l4v1) edge (l3v1);
		        \draw [densely dotted]  (l4v2) edge (l3v1);
			\draw  [densely dotted] (l3v1) edge (l2v2);
			\draw (l2v1) edge (l1v2);
			\draw [densely dotted]  (l2v2) edge (l1v2);
		        \draw (l1v1) edge (z);
		        \draw [densely dotted]  (l1v2) edge (z);
		        \draw [densely dotted]  (z) edge (rt); 
		\end{tikzpicture}
	\end{gather}
\end{pr}

\begin{proof}

The presence of the terms \eqref{ICComLT} restricts the subtrees composed of $\alpha$'s to those growing rightwards (as in the case of the $\Com$ operad).

From \eqref{ICIn1LT} we deduce that $i$ can only be a leaf. The terms \eqref{ICMultLT1}  and \eqref{ICMultLT2} ensure that $l$-subtrees can only grow rightwards, and that we can't graft $\alpha$ upon $l$. Also we can't graft $l$ upon $r$, and, by \eqref{ICIn2LT}, $i$ upon $l$. Gathering this data together we prove the claim.
\end{proof}

Let's make some observations about a tree of the form \eqref{ICBigTree}. First, if it has the output of the second colour, it must have at least one leaf of the second colour. Second, given two numbers $n \ge 0$ and $m > 0$ and a colouring $\chi$ of the set $m + n$ of type $(n, m)$, there is exactly one tree of this type, with the output of the second colour, whose inputs are coloured with $\chi$. Namely, if the first element of the second colour in $\chi$ is not the last element of the set, it corresponds to the only vertex of type $r$ (and the rest of $\chi$ is acquired by grafting or not grafting $i$'s on $\alpha$'s); and if it is the last element in the set, the corresponding tree consist solely of $l$'s. 

Otherwise, if the output of the tree is of the first colour, it means that the tree is composed from $\alpha$'s and $\iota$'s, and there is no restriction for $m$ to be greater than zero (so, any $\chi$ is feasible).

Thus we conclude that for $n + m \ge 2$:
\begin{equation}
\label{eq::grICom}
\begin{array}{c}
\dim\ICom(n + m, \chi, \mathbf{\vdots}) \leq  \dim \gr \ICom (n + m, \chi, \mathbf{\vdots}) =1 ; \\
\dim\ICom(n + m, \chi, \mathbf{|}) \leq  \dim \gr \ICom (n + m, \chi, \mathbf{|}) =\begin{cases}
0, \text{ if } m=0, \\
1, \text{ if } m\geq 1
\end{cases}
\end{array}
\end{equation}
Note that one can easely construct an $\ICom$-algebra consisting of a commutative algebra $A$ and its ideal $I$ such that each operation of the type \eqref{ICBigTree} is different from zero, so our bounds are in fact tight what follows that the defining relations of the shuffle operad $\ICom$ constitute a Gr\"obner basis.
The generating series of $\ICom$ coincides with the generating series of $\gr\ICom$ and are easily computed thanks to~\eqref{eq::grICom} what finishes the proof of Theorem~\ref{thm::ICom}.

\begin{cor}
The operad $\ICom$ is Koszul and its Koszul dual operad admits a quadratic Gr\"obner basis.
\end{cor}	

\subsection{$\AffHS$ operad}
\label{sec::AffHS}

In \cite{Merkulov} Merkulov introduced a notion of \emph{affine homogeneous space}, facilitating the study of deformation theory:

\begin{definition}
An \emph{affine homogeneous space} is a collection of data $(\frkg, \frkh, \langle \, , \, \rangle, \varphi)$ consisting of:
	\begin{itemize}
	\item a Lie algebra $\frkg$ with Lie bracket $[ \, , \, ]$;
	\item a vector space $\frkh$ with a $\frkg$-module structure $\langle \, , \, \rangle \, : \, \frkg \otimes \frkh \to \frkg$;
	\item a linear map $\varphi \, : \, \frkg \to \frkh$, satisfying the equation
		\[
		\varphi([a, b]) = \langle a, \varphi(b) \rangle - (-1)^{|a||b|} \langle b, \varphi (a) \rangle
		\]
	for any $a, b \in \frkg$.
	\end{itemize}
\end{definition}

The operad $\AffHS$ governing affine homogeneous spaces has three generators:

\begin{gather}
\label{AffHSGen}
	\unTreeDS{i}{,}
	\quadTreeSSS{\beta}{1}{2}{= \quad -}
	\quadTreeSSS{\beta}{2}{1}{,}
	\quadTreeDSD{m}{1}{2}{}
\end{gather}

subject to the following set of relations:

\begin{gather}
\label{AffHSLie}
	\cubLeftTreeSSSSS{\beta}{\beta}{1}{2}{3}{-}
	\cubLeftTreeSSSSS{\beta}{\beta}{1}{3}{2}{-}
	\cubRightTreeSSSSS{\beta}{\beta}{1}{2}{3}{}
\end{gather}

\begin{gather}
\label{AffHSMod}
	\cubLeftTreeDSDSS{m}{\beta}{1}{2}{3}{+}
	\cubRightTreeDSDSD{m}{m}{1}{2}{3}{-}
	\cubRightTreeDSDSD{m}{m}{2}{1}{3}{}
\end{gather}

\begin{gather}
\label{AffHSMor}
	\highYTreeDSSS{i}{\beta}{1}{2}{+}
	\JTreeDSDS{m}{i}{1}{2}{-}
	\JTreeDSDS{m}{i}{1}{2}{}
\end{gather}

The corresponding shuffle operad has four generators:

\begin{gather}
\label{AffHSGenSh}
	\unTreeDS{i}{,}
	\quadTreeSSS{\beta}{1}{2}{,}
	\quadTreeDSD{m}{1}{2}{,}
	\quadTreeDDS{n}{1}{2}{}
\end{gather}

We will employ a modification of QM-ordering (see \eqref{QMOrd}), in which $m$ and $n$ play the role of $x$, and $b$ and $i$ play the role of $y$. So our ordering will based of the ordering on monomials in the following algebra:

\[
	A = \Bbbk \langle m, n, \beta, i, q \rangle \Biggr/ \substack{mq - qm, \; \beta q - q \beta, \; \beta m - m\beta q, \\ nq - qm, \; \beta n - n\beta q, \\ iq - qi, \; im - miq, \\ in - niq}
\]

The monomials in $A$ have the following normal form: first comes an $(m, n)$-word of total degree $d_x$, then a $(\beta, i)$-word of total degree $d_y$ and then $q^{d_q}$. Before comparison we present all involved monomials in the normal form.

 To compare two monomials, we first compare their arities. If equal, we compare their $d_x$'s, the monomial with greater $d_x$ is \emph{smaller}. If equal, we compare the $(m, n)$-words lexicographically. If equal, compare the $d_y$'s, the monomial with greater $d_y$ is \emph{greater}. If equal, compare the $(\beta, i)$-words lexicographically. If equal, compare the $d_q$'s, the monomial with greater $d_q$ is \emph{greater}.
 
\begin{theo}
The operad $\AffHS$ admits a quadratic Gr\"obner basis with respect to 	the  aforementioned $QM$-ordering.
\end{theo}
It is not difficult to show that all $S$-polynomials for the set of relations~\eqref{AffHSGen}-\eqref{AffHSGenSh} can be reduced to zero. However, we want to explain  another proof of this result below.

\begin{proof}	
The $QM$-ordering we defined leads to the following choice of the leading terms.

The relation \eqref{AffHSLie} yields:

\begin{gather}
	\label{AffHSLieLT}
	\cubLeftTreeSSSSS{\beta}{\beta}{1}{2}{3}{}
\end{gather}

The relation \eqref{AffHSMod} yields all the trees in the $S_3$ orbit of the first tree in the relation:

\begin{gather}
	\label{AffHSModLT}
	\cubLeftTreeDSDSS{m}{\beta}{1}{2}{3}{,}
	\cubRightTreeDDSSS{n}{\beta}{1}{2}{3}{,}
	\cubLeftTreeDSDSS{m}{\beta}{1}{3}{2}{}
\end{gather}

And the relation \eqref{AffHSMor} yields:
\begin{gather}
	\label{AffHSMorLT}
	\highYTreeDSSS{i}{\beta}{1}{2}{}
\end{gather}

Therefore, the element of the coloured Koszul dual operad $\AffHS^{!}$ are spanned by common multiples of the aforementioned leading monomials. What follows that 
\begin{equation}
\label{eq::AffHS::dim}
\begin{array}{c}
\dim\AffHS^{!}(m,n,|) \leq \dim\gr\AffHS^{!}(m,n,|) = 1, \text{ if } n=0; \\
\dim\AffHS^{!}(m,n,\vdots) \leq \dim\gr\AffHS^{!}(m,n,\vdots) = 1 \text{ if } n\leq 1.
\end{array}
\end{equation}
and $\dim\gr\AffHS^{!}(m,n,\vdots)$ equals zero in all other cases.
Using a particular algebra over the operad $\AffHS^{!}$ one can show that the left hand side in Inequalities~\eqref{eq::AffHS::dim} is bounded from below by $1$ and, therefore, $\AffHS^{!}$ admits a quadratic Gr\"obner basis. Theorem~\ref{thm::PBW::Koszul} implies that the same happens for $\AffHS$. 
\end{proof}

\begin{cor}
	The suboperad of $\AffHS$ generated by $m$ and $i$ is free and there is an isomorphism of coloured symmetric collections:
	$$\AffHS \simeq (\Lie; \cF(m,i)) \Rightarrow F_{\AffHS}(t_1,t_2) = \left(-\ln(1-t_1); \frac{t_1+t_2}{1-t_1}\right)$$
\end{cor}
\begin{proof}
Follows from the description of normal forms. For example,	
	each shuffle monomial in the free operad $\cF(m,i)$ is not divisible by any leading term of the given Gr\"obner basis.
\end{proof}	

\subsection{$\MLie$ operad}
The operad $\MLie$ has two generators:

\begin{gather}
\label{MLieGen}
	\quadTreeDDD{\beta}{1}{2}{= \; -}
	\quadTreeDDD{\beta}{2}{1}{,}
	\quadTreeSDS{d}{1}{2}{}
\end{gather}

subject to the following two relations:

\begin{gather}
\label{MLieJ}
	\cubLeftTreeDDDDD{\beta}{\beta}{1}{2}{3}{-}
	\cubLeftTreeDDDDD{\beta}{\beta}{1}{3}{2}{-}
	\cubRightTreeDDDDD{\beta}{\beta}{1}{2}{3}{}
\end{gather}

\begin{gather}
\label{MLieM}
	\cubLeftTreeSDSDD{d}{\beta}{1}{2}{3}{-}
	\cubRightTreeSDSDS{d}{d}{1}{2}{3}{+}
	\cubRightTreeSDSDS{d}{d}{2}{1}{3}{}
\end{gather}

Algebra over $\MLie$ is a pair of a Lie algebra $L$ and an $L$-module. 
The colouring of inputs/outputs matches the colouring of the inputs of the generators of the operad 
$\LP$ considered in the succeeding Example~\ref{sec::SC}, since the operad $\MLie$ is a suboperad of $\LP$.

The corresponding shuffle operad has three generators:

\begin{gather}
\label{MLieShGen}
	\quadTreeDDD{\beta}{1}{2}{,}
	\quadTreeSDS{d}{1}{2}{,}
	\quadTreeSSD{e}{1}{2}{=}
	\quadTreeSDS{d}{2}{1}{}
\end{gather}

\begin{theo}
	The operad $\MLie$ admits a quadratic Gr\"obner basis with respect to the $QM$-ordering with $d$ and $e$ playing the role of $x$, and $\beta$ playing the role of $y$.
\end{theo}
\begin{proof}
 The $QM$-ordering leads to the following list of leading terms:

The relation \eqref{MLieJ} yields:
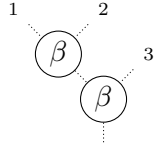
\[
	\cubLeftTreeDDDDD{\beta}{\beta}{1}{2}{3}{}
\]
And the relation \eqref{MLieM} yields:
\[
	\cubLeftTreeSDSDD{d}{\beta}{1}{2}{3}{,}
	\cubLeftTreeSDSDD{d}{\beta}{1}{3}{2}{,}	
	\cubRightTreeSSDDD{e}{\beta}{1}{2}{3}{}
\]
It is well known that the uncoloured $S$-polynomial corresponding to the small common multiple of two Jacobi identities can be reduced to zero. The remaining $S$-polynomial is assigned to the small common multiple of the Jacobi identity~\eqref{MLieJ} and~\eqref{MLieM}: 
We write the small common multiple of the Jacobi relation~\eqref{MLieJ} and module structure~\eqref{MLieM} as well as reductions of the corresponding $S$-polynomial using the language of composition of operations with numbers indexing outputs:
\[
d(B(B(1,2),3),4) :=	\quartLeftTreeSDSDDDD{d}{B}{B}{1}{2}{3}{4}{} 
\]
The corresponding $S$-polynomial $\mu$ is equal to
\begin{multline*} 
\mu:= d(1,4)\circ_1\left[B(B(1,2),3)- B(B(1,3),2) - B(1,B(2,3))\right] \\
 - \left[d(B(1,3),4)- d(1,d(3,4)) + e(d(1,4),3)\right]\circ_1 B(1,2) 
\end{multline*}
We underline all monomials that admits further reduction (rewritings) of the $S$-polynomial $\mu$: 
\begin{gather*}
\mu =  - \underline{d(B(B(1,3),2),4)} - \underline{d(B(1,B(2,3)),4)} + \underline{d(B(1,2),d(3,4))} - \underline{e(d(B(1,2),4),3)} = \\
  \stackrel{\eqref{MLieM}}{=}  
  - {d(B(1,3),d(2,4))} + \underline{e(d(B(1,3),4),2)} - \underline{d(1,d(B(2,3),4))} + e(d(1,4),B(2,3)) +
\\   + d(1,d(2,d(3,4))) - e(d(1,d(3,4)),2) - e(d(1,d(2,4)),3) + e(e(d(1,4),2),3) =  \\
  \stackrel{\eqref{MLieM}}{=}  - \underline{d(B(1,3),d(2,4))} + \underline{e(d(1,4),B(2,3))}  - e(d(1,d(2,4)),3) + e(e(d(1,4),2),3) \\  - e(e(d(1,4),3),2) + d(1,e(d(2,4),3)) =0
\end{gather*} 
All remaining $S$-polynomials corresponds to the action of symmetric group on the colourings of the latter one.
\end{proof}
\begin{cor}
	There is an isomorphism of coloured collections 
	$\MLie \simeq (\cF(d),\Lie)$, where by $\cF(d)$ we denote the free $2$-coloured operad generated
by a single element $\quadTreeSDS{d}{1}{2}{}$.
In particular, 
$$
\overrightarrow{F_{\MLie}}(t_1,t_2) = \left(\frac{t_1}{1 - t_2} + t_2 ; - \log (1 - t_2) \right)
$$
\end{cor}
\begin{proof}
The set of leading monomials explains the structure of normal words in $\MLie$. What follows that $\MLie$ consists of two disjoint parts. The first part is the operad $\Lie$ generated by $\beta$, and the second part is the free shuffle coloured operad generated by $d$ and $e$. The latter is the shuffle operad assigned to the free symmetric coloured operad generated by a single element $d$.
\end{proof}

\subsection{$\HSC$ and $\LP$ operads}
\label{sec::SC}

In \cite{HoefelLivernet} Hoefel and Livernet provide a description of the operad $\SC$ (Swiss Cheese) and its zeroth homology $\HSC$. For the latter the authors proved its Koszulness and provide the Koszul dual operad -- the operad of Leibniz pairs $\LP$. We present a quadratic Gr\"obner basis for the latter operad and hence present another proof of the koszulness of $\HSC$ and $\LP$. Moreover, we computed the generating series of $\LP$ and $\HSC$.

\begin{definition}
A \emph{Leibniz pair} is a pair of a Lie algebra $L$ and an associative algebra $A$ together with a morphism of Lie algebras $L \to \Der (A)$.
\end{definition}

The operad $\LP$ has three generators:

\begin{gather}
\label{LPGen}
	\quadTreeDDD{\beta}{1}{2}{= \; -}
	\quadTreeDDD{\beta}{2}{1}{,}
	\quadTreeSSS{a}{1}{2}{,}
	\quadTreeSDS{d}{1}{2}{}
\end{gather}

subject to the following relation:

The Jacobi relation:

\begin{gather}
\label{LPJ}
	\cubLeftTreeDDDDD{\beta}{\beta}{1}{2}{3}{-}
	\cubLeftTreeDDDDD{\beta}{\beta}{1}{3}{2}{-}
	\cubRightTreeDDDDD{\beta}{\beta}{1}{2}{3}{}
\end{gather}

The associativity relation for $a$:

\begin{gather}
\label{LPA}
	\cubLeftTreeSSSSS{a}{a}{1}{2}{3}{-}
	\cubRightTreeSSSSS{a}{a}{1}{2}{3}{}
\end{gather}

The derivation relation:

\begin{gather}
\label{LPD}
	\cubRightTreeSDSSS{d}{a}{1}{2}{3}{-}
	\cubRightTreeSSSDS{a}{d}{2}{1}{3}{-}
	\cubLeftTreeSSSDS{a}{d}{1}{2}{3}{}
\end{gather}

The Lie algebra morphism relation:

\begin{gather}
\label{LPM}
	\cubLeftTreeSDSDD{d}{\beta}{1}{2}{3}{-}
	\cubRightTreeSDSDS{d}{d}{1}{2}{3}{+}
	\cubRightTreeSDSDS{d}{d}{2}{1}{3}{}
\end{gather}

The corresponding shuffle operad has five generators:

\begin{gather}
\label{LPShGen}
	\quadTreeDDD{\beta}{1}{2}{,}
	\quadTreeSSS{a}{1}{2}{,}
	\quadTreeSSS{b}{1}{2}{=}
	\quadTreeSSS{a}{2}{1}{,}
	\quadTreeSDS{d}{1}{2}{,}
	\quadTreeSSD{e}{1}{2}{=}
	\quadTreeSDS{d}{2}{1}{}
\end{gather}

\begin{theo}
	The defining relations of the operad $\LP$ forms a quadratic Gr\"obner bases of relations with respect to the $QM$ (partial) ordering with $a$ and $b$ being variables of the type $x$, $\beta$, $m$, and $n$ are $y$-type variables, and in addition $a > b$, $m > n$ lexicographically.
\end{theo}
\begin{proof}
The $S$-polynomials for relations $\eqref{LPM}$ and $\eqref{LPJ}$ can be reduced to zero in the same way as in the previous example. We are left to check the reducibility of the $S$-polynomials for $\eqref{LPD}$ and $\eqref{LPA}$. We provide the reduction of one of these polynomials associated with the small common multiple $d(1,a(2,a(3,4)))$, as they lie in a single $\Sg_n$-orbit.

 \begin{align*}&d(1,a(a(2,3),4)) - b(d(1,a(3,4)),2) - \underline{a(d(1,2),a(3,4))} \stackrel{\eqref{LPA}}{=}  \\ 
&\underline{d(1,a(a(2,3),4))} - \underline{b(d(1,a(3,4)),2)} - a(a(d(1,2),3),4) \stackrel{ \eqref{LPD} }{=}  \\ 
& - a(a(d(1,2),3),4) + b(d(1,4),a(2,3)) + \underline{a(d(1,a(2,3)),4)} - b(b(d(1,4),3),2)
 &  - b(a(d(1,3),4),2) \stackrel{\eqref{LPD}}{=}  \\ 
&b(d(1,4),a(2,3)) - b(b(d(1,4),3),2) - b(a(d(1,3),4),2) + \underline{a(b(d(1,3),2),4)} \stackrel{ \eqref{LPA} }{=}  \\ 
&\underline{b(d(1,4),a(2,3))} - b(b(d(1,4),3),2) \stackrel{ \eqref{LPA}}{=} 0\\ 
\end{align*}

\end{proof}	
The $QM$-ordering leads to the following choice of the leading terms.

The relation \eqref{LPJ} yields:

\[
	\cubLeftTreeDDDDD{\beta}{\beta}{1}{2}{3}{}
\]

The relation \eqref{LPA} yields:

\[
	\cubRightTreeSSSSS{a}{a}{1}{2}{3}{,}
	\cubLeftTreeSSSSS{a}{b}{1}{2}{3}{,}
	\cubRightTreeSSSSS{b}{b}{1}{2}{3}{,}
	\cubRightTreeSSSSS{a}{b}{1}{2}{3}{,}
	\cubRightTreeSSSSS{b}{a}{1}{2}{3}{,}
	\cubLeftTreeSSSSS{a}{b}{1}{3}{2}{}
\]

The relation \eqref{LPD} yields:

\[
	\cubRightTreeSDSSS{d}{a}{1}{2}{3}{,}
	\cubLeftTreeSSDSS{e}{a}{1}{3}{2}{,}
	\cubLeftTreeSSDSS{e}{b}{1}{2}{3}{,}
	\cubRightTreeSDSSS{d}{b}{1}{2}{3}{,}
	\cubLeftTreeSSDSS{e}{b}{1}{3}{2}{,}
	\cubLeftTreeSSDSS{e}{a}{1}{2}{3}{}
\]

And the relation \eqref{LPM} yields:

\[
	\cubLeftTreeSDSDD{d}{\beta}{1}{2}{3}{,}
	\cubLeftTreeSSDSD{e}{e}{1}{2}{3}{,}
	\cubLeftTreeSDSDD{d}{\beta}{1}{3}{2}{}	
\]

The choice of the leading terms in the derivation relation \eqref{LPD} ensures that any element of the operad can be rewritten in the following normal form: a two-level tree with the bottom level consists of vertices $a$ and $b$, and the top level consists of the vertices $d$, $e$, and $\beta$. 

From this observation and the generating relations of $\LP$ we conclude that $\LP = \As \circ \MLie$, as symmetric coloured collections, where $\As$ is the associative operad generated by $a$, the operad $\MLie$ was described in the previous example.

Now we can compute the generating series for $\LP$:

\begin{multline*}
{{\LP} \simeq {\As}\circ \MLie  \Rightarrow  
\overrightarrow{F_{\LP}}(t_1,t_2) = \overrightarrow{F_{\As}}\circ \overrightarrow{F_{\MLie}} = }
\\ = \left(\frac{t_1}{1-t_1}; t_2\right)\circ 
\left(\frac{t_1}{1 - t_2} + t_2 ; - \log{(1 - t_2)} \right) = \left( \frac{-t_1}{t_1 + t_2}; - \log (1 - t_2)\right).
\end{multline*}

\subsection{$\DCom$ operad}
\label{subSecDCom}
Pairs of the form $(A,D)$, consisting of a  commutative algebra $A$ and a space $D$ of its derivation are governed by the following two-coloured operad $\DCom$ generated by two generators of arity 2, $\alpha$ and $d$:	

\begin{gather}
\label{DComGen}
		\quadTreeSSS{\al}{1}{2}{=}
		\quadTreeSSS{\al}{2}{1}{,}
		\quadTreeSSD{d}{1}{2}{}
\end{gather}

with the relations for $\alpha$ being an associative commutative product and $d$ being the derivation of $\alpha$ (the Leibniz rule):

\begin{gather}
\label{DComAs}
	\cubLeftTreeSSSSS{a}{a}{1}{2}{3}{-}
	\cubRightTreeSSSSS{a}{a}{1}{2}{3}{}
\\
\label{DComLeib}
	\cubLeftTreeSSDSS{d}{\al}{1}{2}{3}{=}
	\cubRightTreeSSSSD{\al}{d}{1}{2}{3}{+}
	\cubLeftTreeSSSSD{\al}{d}{1}{3}{2}{}
\end{gather}
This example is rather contrived, as such pairs fit more naturally in the uncoloured framework. We consider this two-coloured version with a view to use the computations for it in Example~\ref{subSecLR}.

\begin{theo}
	The $2$-coloured operad $\DCom$ admits a quadratic Gr\"obner basis with respect to the $QM$-ordering described in \S\ref{QMOrd} with $y = d, x = \alpha$.
	
There is an isomorphism of coloured symmetric collections $\DCom = \Com \circ \cF(d)$, where $\cF(d)$ is the free operad generated by $d$ and $\Com$ consists of operations of the first colour and the generating series $\overrightarrow{F_{\DCom}}(t_1,t_2)$ is equal to $\left(e^{\frac{t_1}{1-t_2}},t_2\right)$.
\end{theo}
\begin{proof}
The corresponding shuffle operad has three generators:

	\begin{gather}
	\label{DComShGen}
		\quadTreeSSS{\al}{1}{2}{,}
		\quadTreeSSD{d}{1}{2}{,}
		\quadTreeSDS{e}{1}{2}{=}
		\quadTreeSSD{d}{2}{1}{}
	\end{gather}

The following list of leading terms do appear with respect to the aforementioned $QM$-ordering:

\begin{gather}
	\cubLeftTreeSSSSS{\al}{\al}{1}{3}{2}{,}
	\cubRightTreeSSSSS{\al}{\al}{1}{2}{3}{,}
	\cubRightTreeSDSSS{e}{\al}{1}{2}{3}{,}
	\cubLeftTreeSSDSS{d}{\al}{1}{3}{2}{,}
	\cubLeftTreeSSDSS{d}{\al}{1}{2}{3}{}
\end{gather}
Thus, there are the $S$-polynomials of the first colour that deal with the commutative associative product and are known to be reducible to $0$ and there is an $S$-polynomial associated with the small common multiple $e(1,\alpha(\alpha(2,4),3))$ of Relations~\eqref{DComAs} and \eqref{DComLeib}:

 \begin{align*}&e(1,\alpha(\alpha(2,3),4)) - \alpha(e(1,\alpha(2,4)),3) - \underline{\alpha(e(1,3),\alpha(2,4))} \stackrel{ \eqref{DComAs} }{=}  \\ 
&\underline{e(1,\alpha(\alpha(2,3),4))} - \underline{\alpha(e(1,\alpha(2,4)),3)} - \alpha(\alpha(e(1,3),2),4) \stackrel{ \eqref{DComLeib} }{=}  \\ 
& - \alpha(\alpha(e(1,3),2),4) + \underline{\alpha(e(1,\alpha(2,3)),4)} + \alpha(e(1,4),\alpha(2,3)) - \alpha(\alpha(e(1,2),4),3)\\ 
 &  - \alpha(\alpha(e(1,4),2),3) \stackrel{ \eqref{DComLeib} }{=}  \\ 
&\alpha(e(1,4),\alpha(2,3)) - \underline{\alpha(\alpha(e(1,2),4),3)} - \alpha(\alpha(e(1,4),2),3) + \alpha(\alpha(e(1,2),3),4) \stackrel{ \eqref{DComAs} }{=}  \\ 
&\underline{\alpha(e(1,4),\alpha(2,3))} - \alpha(\alpha(e(1,4),2),3) \stackrel{ \eqref{DComAs} }{=}  0\\ 
\end{align*}

All other $S$ polynomials differ from this one by the action of symmetric group what affects the replacement $d$ by $e$.

It is immediate to see that the elements of $\DCom$ have the following normal form: a leftward growing tree of $\al$'s with arbitrary compositions of $d$ and $e$ plugged into it. It means that as symmetric coloured collections $\DCom = \Com \circ \cF(d)$, where $\cF(d)$ is the free operad generated by $d$ and $\Com$ consists of operations of the first colour.

Now we can compute the generating series for $\DCom$:
$$
\overrightarrow{F_{\DCom}}(t_1,t_2) = (e^{t_1}-1;t_2) \circ (\frac{t_1}{1-t_2};t_2) = \left(e^{\frac{t_1}{1-t_2}}-1; t_2\right)
$$
\end{proof}

\subsection{$\LieR$ operad}
\label{subSecLR}
Following~\cite{Rinehart}, ~\cite{KordonLambre} we say that a \emph{Lie-Rinehart} algebra is a pair $(S, L)$ of a commutative algebra $S$ and a Lie algebra $L$, such that $L$ acts on $S$ by derivations, $L$ is an $S$ module, and the following relations hold:
\[
(s \alpha) (t) = s \cdot (\alpha(t)),
\hspace{35pt}
[\alpha, s\beta] = s [\alpha, \beta] + \alpha(s) \beta;
\]
for $s, t \in S$, $\alpha, \beta \in L$.
With each algebraic variety or smooth manifold $X$ one can assign a \emph{Lie-Rinehart} algebra consisting of the commutative algebra of functions on $X$ and the Lie algebra of vector fields on $X$.

We define the operad $\LieR$ as the coloured symmetric operad on two colours generated by four operations:

	\begin{gather}
	\label{LRGen}
		\quadTreeSSS{\al}{1}{2}{=}
		\quadTreeSSS{\al}{2}{1}{,}
		\quadTreeSSS{\beta}{1}{2}{= \quad -}
		\quadTreeSSS{\beta}{2}{1}{,}
		\quadTreeSSD{d}{1}{2}{,}
		\quadTreeDSD{m}{1}{2}{}
	\end{gather}

subject to the following list of relations.

The associativity relation for $\al$:

\begin{gather}
\label{LRAs}
	\cubLeftTreeSSSSS{a}{a}{1}{2}{3}{-}
	\cubRightTreeSSSSS{a}{a}{1}{2}{3}{}
\end{gather}

The Jacobi relation for $\beta$:
\begin{gather}
\label{LRJac}
	\cubLeftTreeDDDDD{\beta}{\beta}{1}{2}{3}{-}
	\cubLeftTreeDDDDD{\beta}{\beta}{1}{3}{2}{-}
	\cubRightTreeDDDDD{\beta}{\beta}{1}{2}{3}{}
\end{gather}

The Leibniz rule:

\begin{gather}
\label{LRLeib}
	\cubLeftTreeSSDSS{d}{\al}{1}{2}{3}{=}
	\cubRightTreeSSSSD{\al}{d}{1}{2}{3}{+}
	\cubLeftTreeSSSSD{\al}{d}{1}{3}{2}{}
\end{gather}

The relation describing the morphism of Lie algebras $L \to \Der (S)$:
\begin{gather}
\label{LRMor}
	\cubRightTreeSSDDD{d}{\be}{1}{2}{3}{=}
	\cubLeftTreeSSDSD{d}{d}{1}{3}{2}{-}
	\cubLeftTreeSSDSD{d}{d}{1}{2}{3}{}
\end{gather}

The relation stating that $L$ is an $S$-module:

\begin{gather}
\label{LRSMod}
	\cubLeftTreeDSDSS{m}{\al}{1}{2}{3}{=}
	\cubRightTreeDSDSD{m}{m}{1}{2}{3}{}
\end{gather}

And two relation specific for the Lie-Rinehart algebras:

\begin{gather}
\label{LRA}
	\cubRightTreeSSDSD{d}{m}{1}{2}{3}{=}
	\cubLeftTreeSSSSD{\al}{d}{1}{3}{2}{}
\end{gather}

\begin{gather}
\label{LRB}
	\cubLeftTreeDDDSD{\be}{m}{1}{3}{2}{=}
	\cubRightTreeDSDDD{m}{\be}{1}{2}{3}{+}
	\cubLeftTreeDSDSD{m}{d}{1}{2}{3}{}
\end{gather}

The corresponding shuffle operad has six generators:

	\begin{gather}
	\label{LRShGen}
		\quadTreeSSS{\al}{1}{2}{,}
		\quadTreeSSS{\beta}{1}{2}{,}
		\quadTreeSSD{d}{1}{2}{,}
		\quadTreeSDS{e}{1}{2}{=}
		\quadTreeSSD{d}{2}{1}{,}
		\quadTreeDSD{m}{1}{2}{,}
		\quadTreeDDS{n}{1}{2}{=}
		\quadTreeDSD{m}{2}{1}{}
	\end{gather}

\begin{theo}	
\begin{itemize}
\item 
	The operad $\LieR$ admits a quadratic Gr\"obner basis;
\item	
The $2$-coloured symmetric collection $\LieR$ is isomorphic to the composition  $N_m\circ(\DCom^1;\Lie)$. Here $\DCom^1$ is the subset of $\DCom$ spanned by operations with the output of the first (straight) colour and  $N_m$ is a nilpotent quadratic operad generated by a single element $\quadTreeDSD{m}{1}{2}{}$ subject to the relation $\cubRightTreeDSDSD{m}{m}{1}{2}{3}{=0}.$
\end{itemize}	
\end{theo}
\begin{proof}
In this example we employ a further modification of a QM-ordering. We divide the generators into three groups: light ($\be$, $n$, $m$), heavy ($d$, $e$), and superheavy ($\al$). The light generators play the role of $y$, the heavy ones play the role of $x$, and the superheavy play the role of $x$ relatively to the heavy ones. Namely, we base our ordering on the ordering of the monomials in the algebra $A$ defined as:
\[
	A = \Bbbk \langle \al, d, e, m, n, \be, q \rangle \Biggr/ \substack{\al q - q \al, \; d \al - \al d q, \; e \al - \al e q,
\\ m \al - \al m q, \; n \al - \al n q, \; \be \al - \al \be q,
\\ d q - q d, \; m d - d m q, \; nd - dnq, \; \be d - d \be q,
\\  e q - q e, \; m e - e m q, \; ne - enq, \; \be e - e \be q,
\\ mq - qm, \; nq - qn, \; \be q - q \be
}
\]

A word in $A$ has the following normal form: 
\[
w = w_{\al} w_{d,e} w_{m, n, \be} q^{d_q},
\]
where  $w_{\al}$ is an $\al$-word of degree $d_{\al}$, $w_{d,e}$ is a $(d, e)$-word of degree $d_{(d,e)}$, and $w_{m, n, \be}$ is an $(m, n, \be)$-word of degree $d_{(m, n, \be)}$. To compare two such words, first compare $d_{\al}$, the word with smaller $d_{\al}$ is greater. If equal, the word with smaller  $d_{(d,e)}$ is greater. If equal, the word with greater  $d_{(m, n, \be)}$ is greater. If equal, the word with greater $d_q$ is greater.

This ordering leads to the following choice of the leading terms:
The relation \eqref{LRAs} yields:

\begin{gather}
	\cubLeftTreeSSSSS{\al}{\al}{1}{3}{2}{,}
	\cubRightTreeSSSSS{\al}{\al}{1}{2}{3}{}
\end{gather}

The relation \eqref{LRJac} yields:

\begin{gather}
	\cubLeftTreeDDDDD{\be}{\be}{1}{2}{3}{}
\end{gather}

The relation \eqref{LRLeib} yields:

\begin{gather}
	\cubRightTreeSDSSS{e}{\al}{1}{2}{3}{,}
	\cubLeftTreeSSDSS{d}{\al}{1}{3}{2}{,}
	\cubLeftTreeSSDSS{d}{\al}{1}{2}{3}{}
\end{gather}

The relation \eqref{LRSMod} yields:

\begin{gather}
	\cubRightTreeDSDSD{m}{m}{1}{2}{3}{,}
	\cubLeftTreeDDSSD{n}{m}{1}{3}{2}{,}
	\cubRightTreeDSDDS{m}{n}{1}{2}{3}{,}
	\cubLeftTreeDDSDS{n}{n}{1}{2}{3}{,}
	\cubLeftTreeDDSDS{n}{n}{1}{3}{2}{,}
	\cubLeftTreeDDSSD{n}{m}{1}{2}{3}{}
\end{gather}

The relation \eqref{LRMor} yields:

\begin{gather}
	\cubRightTreeSSDDD{d}{\be}{1}{2}{3}{,}
	\cubLeftTreeSDSDD{e}{\be}{1}{3}{2}{,}
	\cubLeftTreeSDSDD{e}{\be}{1}{2}{3}{}
\end{gather}

The relation \eqref{LRA} yields:

\begin{gather}
	\cubLeftTreeSDSSD{e}{m}{1}{2}{3}{,}
	\cubLeftTreeSDSDS{e}{n}{1}{2}{3}{,}
	\cubRightTreeSSDDS{d}{n}{1}{2}{3}{,}
	\cubLeftTreeSDSSD{e}{m}{1}{3}{2}{,}
	\cubRightTreeSSDSD{d}{m}{1}{2}{3}{,}
	\cubLeftTreeSDSDS{e}{n}{1}{3}{2}{}
\end{gather}

And the relation \eqref{LRB} yields:

\begin{gather}
	\cubRightTreeDDDSD{\be}{m}{1}{2}{3}{,}
	\cubLeftTreeDDDSD{\be}{m}{1}{3}{2}{,}
	\cubLeftTreeDDDDS{\be}{n}{1}{2}{3}{,}
	\cubRightTreeDDDDS{\be}{n}{1}{2}{3}{,}
	\cubLeftTreeDDDDS{\be}{n}{1}{3}{2}{,}
	\cubLeftTreeDDDSD{\be}{m}{1}{2}{3}{}
\end{gather}

Note that for every relation excluding the Jacobi identity for $\beta$ and associativity relations for $\alpha$ all the leading terms constitute the entire $\Sigma_3$-orbit acting on different colourings of inputs. This observation shortens the number of $S$-polynomials whose reductions one has to verify.
We work out all reductions (one for each $\Sigma_3$-orbit) in the Appendix~\ref{sec::LR::computations}. 

Contemplating on the choice of the leading terms, one can conclude, that the elements of the operad $\LieR$ have the following normal form:
\begin{itemize}
	\item On the first level they have one vertex of type $n$ or $m$ (or none of those)
	\item Two blocks can be grafted on $n$ or $m$: block consisting of $\alpha$'s ($\Com$-block) and block consisting of $\beta$'s ($\Lie$-block).
	\item Additionally, any free input of the $\Com$-block may be decorated with an arbitrary $(d,e)$-tree.
\end{itemize}
	
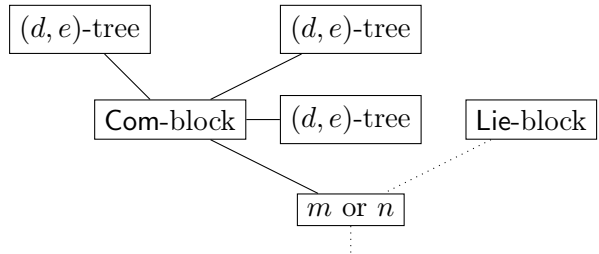
\begin{gather}
		\begin{tikzpicture}[scale=0.6]
		    \node[rectangle,draw, minimum size=1pt] (z) {$m$ or $n$};
		    \node[rectangle, draw, minimum size=1pt] (l1v1) at (-4,2) {$\Com$-block} ;
		    \node[rectangle, draw, minimum size=1pt] (l1v2) at (0,2) {$(d,e)$-tree} ;
		    \node[rectangle,draw, minimum size=1pt] (l1v3) at (4,2) {$\Lie$-block} ;
		    \node[rectangle,draw, minimum size=1pt] (l2v1) at (-6,4) {$(d,e)$-tree} ;
		    \node[rectangle,draw, minimum size=1pt] (l2v2) at (0,4) {$(d,e)$-tree} ;
		    \coordinate (rt) at (0,-1);
			\draw (l2v1) edge (l1v1);
			\draw (l2v2) edge (l1v1);
		        \draw (l1v1) edge (z);
		        \draw (l1v3)[dotted] edge (z);
		        \draw (l1v2) edge (l1v1);
		        \draw[dotted]  (z) edge (rt); 
		\end{tikzpicture}
	\end{gather}

So the elements of the corresponding symmetric operad have the form:

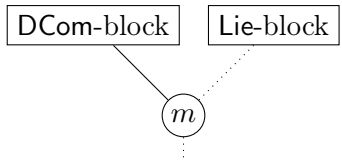
\begin{gather}
		\begin{tikzpicture}[scale=0.6]
		    \node[ver] (z) {$m$};
		    \node[rectangle, draw, minimum size=1pt] (l1v1) at (-2,2) {$\DCom$-block} ;
		    \node[rectangle,draw, minimum size=1pt] (l1v2) at (2,2) {$\Lie$-block} ;
		    \coordinate (rt) at (0,-1);
		        \draw (l1v1) edge (z);
		        \draw (l1v2)[dotted] edge (z);
		        \draw[dotted]  (z) edge (rt); 
		\end{tikzpicture}
\end{gather}
Note that the action of the permutation $\sigma$ on the set of inputs of a normal word will give again a normal word whenever $\sigma$ will not interact with the Lie block. On the other hand, the action of symmetric group on the operad Lie is also well known. Thus, we conclude that the description of the normal words implies the isomorphism of the coloured symmetric collections $\LieR$ and the composition 
$N_m\circ (\DCom^1,\Lie)$. Here we denote by $N_m$ the $2$-coloured operad generated by a single operation $m$ in arity $2$ and all non-trivial compositions are equal to zero. The coloured symmetric collection assigned with $N_m$ consists of $m$ in arity $2$, two identity elements of two colours in arity $1$, and $0$ in all other arities.  

This allows us to compute the generating series:
 \[
 \overrightarrow{F_{\LieR}}(t_1, t_2) =(t_1; t_2+ t_1 t_2)\circ  \left(e^{\frac{t_1}{1-t_2}} - 1; - \log (1 - t_2)\right) = \left( e^{\frac{t_1}{1-t_2}} - 1;  - \log (1 - t_2)\cdot e^{\frac{t_1}{1-t_2}} \right)
\]
\end{proof}

\begin{cor}
	The map of coloured operads $\DCom\rightarrow\LieR$ is an embedding.
\end{cor}

\subsection{$\DerCom$ operad}
The $\DerCom$ operad is an operad governing pairs of a commutative algebra $S$ and a Lie algebra $L$, such that $L$ acts on $S$ by derivations and $L$ has a structure of $S$-module.
In combinatorial terms the operad $\DerCom$ is an operad on $2$ colours $\{c,l\}$, is generated by the following list of binary operations:
\begin{itemize}
	\item $\alpha(\ttt,\ttt)\in \DerCom(2,0,c)$ -- a commutative associative product;
	\item $[\ttt,\ttt]\in \DerCom(0,2,l)$ -- a Lie bracket of the derivations, yielding the Jacobi identity;
	\item $d(\ttt,\ttt)\in \DerCom(1,1,c)$ -- the action of the derivation on the elements of a commutative algebra;
	\item $m(\ttt,\ttt)\in \DerCom(1,1,l)$ -- the action of a commutative algebra on the Lie algebra of derivations.
\end{itemize}

It is immediate to see that $\DerCom$ is essentially the operad $\LieR$ without two relations \eqref{LRA} and \eqref{LRB}. 

\begin{theo}
The same choice of ordering and the leading terms also leads to a quadratic Gr\"obner basis for $\DerCom$.
\end{theo}
\begin{proof}
The set of $S$-polynomials for $\DerCom$ is the subset of  $S$-polynomials computed for $\LieR$ and the corresponding reductions do not involve relations \eqref{LRA} and \eqref{LRB} as one can see from the computations presented in Appendix~\ref{sec::LR::computations}. 
\end{proof}	

\appendix

\section{$S$-polynomial for $\LieR$}
\label{sec::LR::computations}
In the appendices we provide a sample of computations for the operads $\LP$ and $\LieR$. The full computation is too voluminous to include here, but by the merit of  $\Sigma_n$-symmetry of the set of the leading terms, it suffice to provide one example for every pair of relations with a non-trivial $S$-polynomial.

The corresponding shuffle coloured operad has the following list of relations where
 we underline the leading monomial in each relation: 
\begin{equation}
\tag{Com1}
\alpha(\alpha(1,2),3)\underline{ - \alpha(\alpha(1,3),2)} = 0
\end{equation}
\begin{equation}
~\tag{Com2}
\alpha(\alpha(1,2),3)\underline{ - \alpha(1,\alpha(2,3))} = 0
\end{equation}
\begin{equation}
\tag{Lie}
\underline{\beta(\beta(1,2),3)} - \beta(\beta(1,3),2) - \beta(1,\beta(2,3)) = 0
\end{equation}
\begin{equation}
\tag{Leib1}
\underline{e(1,\alpha(2,3))} - \alpha(e(1,2),3) - \alpha(e(1,3),2) = 0
\end{equation}
\begin{equation}
\tag{Leib2}
\underline{d(\alpha(1,3),2)} - \alpha(d(1,2),3) - \alpha(1,e(2,3)) = 0
\end{equation}
\begin{equation}
\tag{Leib3}
\underline{d(\alpha(1,2),3)} - \alpha(1,d(2,3)) - \alpha(d(1,3),2) = 0
\end{equation}
\begin{equation}
\tag{Mor1}
\underline{d(1,\beta(2,3))} + d(d(1,2),3) - d(d(1,3),2) = 0
\end{equation}
\begin{equation}
\tag{Mor2}
\underline{e(\beta(1,3),2)} + d(e(1,2),3) - e(1,d(2,3)) = 0
\end{equation}
\begin{equation}
\tag{Mor3}
\underline{ - d(1,\beta(2,3))} + d(d(1,3),2) - d(d(1,2),3) = 0
\end{equation}
\begin{equation}
\tag{Mor4}
\underline{ - e(\beta(1,2),3)} + e(1,e(2,3)) - d(e(1,3),2) = 0
\end{equation}
\begin{equation}
\tag{Mor5}
\underline{ - e(\beta(1,3),2)} + e(1,d(2,3)) - d(e(1,2),3) = 0
\end{equation}
\begin{equation}
\tag{Mor6}
\underline{e(\beta(1,2),3)} + d(e(1,3),2) - e(1,e(2,3)) = 0
\end{equation}
\begin{equation}
\tag{SMod1}
m(\alpha(1,2),3)\underline{ - m(1,m(2,3))} = 0
\end{equation}
\begin{equation}
\tag{SMod2}
m(\alpha(1,2),3)\underline{ - n(m(1,3),2)} = 0
\end{equation}
\begin{equation}
\tag{SMod3}
m(\alpha(1,3),2)\underline{ - m(1,n(2,3))} = 0
\end{equation}
\begin{equation}
\tag{SMod4}
n(1,\alpha(2,3))\underline{ - n(n(1,2),3)} = 0
\end{equation}
\begin{equation}
\tag{SMod5}
n(1,\alpha(2,3))\underline{ - n(n(1,3),2)} = 0
\end{equation}
\begin{equation}
\tag{SMod6}
m(\alpha(1,3),2)\underline{ - n(m(1,2),3)} = 0
\end{equation}
\begin{equation}
\tag{LR-A1}
\underline{e(m(1,2),3)} - \alpha(1,e(2,3))= 0
\end{equation}
\begin{equation}
\tag{LR-A2}
\underline{e(n(1,2),3)} - \alpha(e(1,3),2)= 0
\end{equation}
\begin{equation}
\tag{LR-A3}
\underline{d(1,n(2,3))} - \alpha(d(1,2),3)= 0
\end{equation}
\begin{equation}
\tag{LR-A4}
\underline{e(m(1,3),2)} - \alpha(1,d(2,3))= 0
\end{equation}
\begin{equation}
\tag{LR-A5}
\underline{d(1,m(2,3))} - \alpha(d(1,3),2)= 0
\end{equation}
\begin{equation}
\tag{LR-A6}
\underline{e(n(1,3),2)} - \alpha(e(1,2),3)= 0
\end{equation}
\begin{equation}
\tag{LR-B1}
\label{LR-B1}
\underline{\beta(1,m(2,3))} - m(e(1,2),3) - n(\beta(1,3),2) = 0
\end{equation}
\begin{equation}
\tag{LR-B2}
\underline{ - \beta(m(1,3),2)} - m(d(1,2),3) - m(1,\beta(2,3)) = 0
\end{equation}
\begin{equation}
\tag{LR-B3}
\underline{ - \beta(n(1,2),3)} - n(1,d(2,3)) + n(\beta(1,3),2) = 0
\end{equation}
\begin{equation}
\tag{LR-B4}
\underline{\beta(1,n(2,3))} - m(e(1,3),2) - n(\beta(1,2),3) = 0
\end{equation}
\begin{equation}
\tag{LR-B5} 
\underline{ - \beta(n(1,3),2)} - n(1,e(2,3)) + n(\beta(1,2),3) = 0
\end{equation}
\begin{equation}
\tag{LR-B6}
\underline{ - \beta(m(1,2),3)} - m(d(1,3),2) + m(1,\beta(2,3)) = 0
\end{equation}

 Below is the list of representatives of $\Sigma_4$-orbits of the set of all reductions. 

 Reduction of the $S$-polynomial for relations LR-B1 and LR-B2 associated with the small common multiple $\beta(m(1,4),m(2,3))$:
 \begin{align*} 
& - m(e(m(1,4),2),3) - \underline{n(\beta(m(1,4),3),2)} - m(d(1,m(2,3)),4) - m(1,\beta(m(2,3),4)) \stackrel{ LR-B2 }{=}  \\ 
& - m(e(m(1,4),2),3) - m(d(1,m(2,3)),4) - \underline{m(1,\beta(m(2,3),4))} + n(m(d(1,3),4),2)\\ 
 &  + n(m(1,\beta(3,4)),2) \stackrel{ LR-B6 }{=}  \\ 
& - \underline{m(e(m(1,4),2),3)} - m(d(1,m(2,3)),4) + n(m(d(1,3),4),2) + n(m(1,\beta(3,4)),2)\\ 
 &  + m(1,m(d(2,4),3)) - m(1,m(2,\beta(3,4))) \stackrel{ LR-A4 }{=}  \\ 
& - \underline{m(d(1,m(2,3)),4)} + n(m(d(1,3),4),2) + n(m(1,\beta(3,4)),2) + m(1,m(d(2,4),3))\\ 
 &  - m(1,m(2,\beta(3,4))) - m(\alpha(1,d(2,4)),3) \stackrel{ LR-A5 }{=}  \\ 
&n(m(d(1,3),4),2) + n(m(1,\beta(3,4)),2) + \underline{m(1,m(d(2,4),3))} - \underline{m(1,m(2,\beta(3,4)))}\\ 
 &  - m(\alpha(1,d(2,4)),3) - m(\alpha(d(1,3),2),4) \stackrel{ SMod1 }{=}  \\ 
&\underline{n(m(d(1,3),4),2)} + \underline{n(m(1,\beta(3,4)),2)} - m(\alpha(d(1,3),2),4) - m(\alpha(1,2),\beta(3,4)) \stackrel{ SMod2 }{=}  0\\ 
\end{align*} 
 Reduction of the $S$-polynomial for relations LR-B1 and LR-B2 associated with the small common multiple $ d(1,\beta(2,m(3,4)))$: 
 \begin{align*} &- d(1,m(e(2,3),4)) - \underline{d(1,n(\beta(2,4),3))} - d(d(1,2),m(3,4)) + d(d(1,m(3,4)),2) \stackrel{ LR-A3 }{=}  \\ 
& - \underline{d(1,m(e(2,3),4))} - \underline{d(d(1,2),m(3,4))} + \underline{d(d(1,m(3,4)),2)} - \alpha(d(1,\beta(2,4)),3) \stackrel{ LR-A5 }{=}  \\ 
& - \underline{\alpha(d(1,\beta(2,4)),3)} - \alpha(d(1,4),e(2,3)) - \alpha(d(d(1,2),4),3) + d(\alpha(d(1,4),3),2) \stackrel{ Mor1 }{=}  \\ 
& - \alpha(d(1,4),e(2,3)) + \underline{d(\alpha(d(1,4),3),2)} - \alpha(d(d(1,4),2),3) \stackrel{ Leib2 }{=}  0\\ 
\end{align*} 
 
{Reducing S-polynomial for relations LR-B1 and SMod1 associated with small common multiple} $\beta(1,m(2,m(3,4)))$: 
\begin{align*}& - m(e(1,2),m(3,4)) - \underline{n(\beta(1,m(3,4)),2)} + \underline{\beta(1,m(\alpha(2,3),4))} \stackrel{ LR-B1 }{=}  \\ 
& - \underline{m(e(1,2),m(3,4))} - n(m(e(1,3),4),2) - n(n(\beta(1,4),3),2) + m(e(1,\alpha(2,3)),4)\\ 
 &  + n(\beta(1,4),\alpha(2,3)) \stackrel{ SMod1 }{=}  \\ 
& - \underline{n(m(e(1,3),4),2)} - n(n(\beta(1,4),3),2) + m(e(1,\alpha(2,3)),4) + n(\beta(1,4),\alpha(2,3))\\ 
 &  - m(\alpha(e(1,2),3),4) \stackrel{ SMod2 }{=}  \\ 
& - \underline{n(n(\beta(1,4),3),2)} + m(e(1,\alpha(2,3)),4) + n(\beta(1,4),\alpha(2,3)) - m(\alpha(e(1,2),3),4)\\ 
 &  - m(\alpha(e(1,3),2),4) \stackrel{ SMod5 }{=}  \\ 
&\underline{m(e(1,\alpha(2,3)),4)} - m(\alpha(e(1,2),3),4) - m(\alpha(e(1,3),2),4) \stackrel{ Leib1 }{=}  0\\ 
\end{align*} 
 
{Reducing S-polynomial for relations LR-B1 and Lie} associated with the small common multiple $ \beta(\beta(1,2),m(3,4))$: 
\begin{align*}& - m(e(\beta(1,2),3),4) - n(\beta(\beta(1,2),4),3) + \underline{\beta(\beta(1,m(3,4)),2)} + \underline{\beta(1,\beta(2,m(3,4)))} \stackrel{ LR-B1 }{=}  \\ 
& - m(e(\beta(1,2),3),4) - n(\beta(\beta(1,2),4),3) + \beta(m(e(1,3),4),2) + \beta(n(\beta(1,4),3),2)\\ 
 &  + \underline{\beta(1,m(e(2,3),4))} + \beta(1,n(\beta(2,4),3)) \stackrel{ LR-B1 }{=}  \\ 
& - m(e(\beta(1,2),3),4) - n(\beta(\beta(1,2),4),3) + \underline{\beta(m(e(1,3),4),2)} + \beta(n(\beta(1,4),3),2)\\ 
 &  + \beta(1,n(\beta(2,4),3)) + m(e(1,e(2,3)),4) + n(\beta(1,4),e(2,3)) \stackrel{ LR-B2 }{=}  \\ 
& - m(e(\beta(1,2),3),4) - n(\beta(\beta(1,2),4),3) + \beta(n(\beta(1,4),3),2) + \underline{\beta(1,n(\beta(2,4),3))}\\ 
 &  + m(e(1,e(2,3)),4) + n(\beta(1,4),e(2,3)) - m(d(e(1,3),2),4) - m(e(1,3),\beta(2,4)) \stackrel{ LR-B4 }{=}  \\ 
& - m(e(\beta(1,2),3),4) - n(\beta(\beta(1,2),4),3) + \underline{\beta(n(\beta(1,4),3),2)} + m(e(1,e(2,3)),4)\\ 
 &  + n(\beta(1,4),e(2,3)) - m(d(e(1,3),2),4) + n(\beta(1,\beta(2,4)),3) \stackrel{ LR-B5 }{=}  \\ 
& - \underline{m(e(\beta(1,2),3),4)} - n(\beta(\beta(1,2),4),3) + m(e(1,e(2,3)),4) - m(d(e(1,3),2),4)\\ 
 &  + n(\beta(1,\beta(2,4)),3) + n(\beta(\beta(1,4),2),3) \stackrel{ Mor4 }{=}  \\ 
& - \underline{n(\beta(\beta(1,2),4),3)} + n(\beta(1,\beta(2,4)),3) + n(\beta(\beta(1,4),2),3) \stackrel{ Lie }{=}  0\\ 
\end{align*} 
 
{Reducing S-polynomial for relations LR-A1 and SMod1} associated associated with the small common multiple  $ e(m(1,m(2,3)),4)  
$: 
\begin{align*}& - \underline{\alpha(1,e(m(2,3),4))} + \underline{e(m(\alpha(1,2),3),4)} \stackrel{ LR-A1 }{=}   
 - \underline{\alpha(1,\alpha(2,e(3,4)))} + \alpha(\alpha(1,2),e(3,4)) \stackrel{ Com2 }{=}  0\\ 
\end{align*} 
 
{Reducing S-polynomial for relations LR-A1 and Leib1} associated with the small common multiple $e(m(1,2),\alpha(3,4))$: 
\begin{align*}& - \alpha(1,e(2,\alpha(3,4))) + \underline{\alpha(e(m(1,2),3),4)} + \underline{\alpha(e(m(1,2),4),3)} \stackrel{ LR-A1 }{=}  \\ 
& - \underline{\alpha(1,e(2,\alpha(3,4)))} + \alpha(\alpha(1,e(2,3)),4) + \alpha(\alpha(1,e(2,4)),3) \stackrel{ Leib1 }{=}  \\ 
&\alpha(\alpha(1,e(2,3)),4) + \alpha(\alpha(1,e(2,4)),3) - \underline{\alpha(1,\alpha(e(2,3),4))} - \underline{\alpha(1,\alpha(e(2,4),3))} \stackrel{ Com2 }{=}  0\\ 
\end{align*} 
 {Reducing S-polynomial for relations Mor1 and Lie} associated with the small common multiple $ d(1,\beta(\beta(2,3),4)) 
$: \begin{align*}&\underline{d(d(1,\beta(2,3)),4)} - \underline{d(d(1,4),\beta(2,3))} + \underline{d(1,\beta(\beta(2,4),3))} + \underline{d(1,\beta(2,\beta(3,4)))} \stackrel{ Mor1 }{=}  \\ 
& - d(d(d(1,2),3),4) + d(d(d(1,3),2),4) + d(d(d(1,4),2),3) - d(d(d(1,4),3),2)\\ 
 &  - \underline{d(d(1,\beta(2,4)),3)} + \underline{d(d(1,3),\beta(2,4))} - \underline{d(d(1,2),\beta(3,4))} + \underline{d(d(1,\beta(3,4)),2)} \stackrel{ Mor1 }{=}  0\\ 
\end{align*} 
 
{Reducing S-polynomial for relations Mor1 and Leib2} associated with the small common multiple $d(\alpha(1,4),\beta(2,3)) 
$: \begin{align*}
&d(d(\alpha(1,4),2),3) - d(d(\alpha(1,4),3),2) + \underline{\alpha(d(1,\beta(2,3)),4)} + \alpha(1,e(\beta(2,3),4)) \stackrel{ Mor1 }{=}  \\ 
&d(d(\alpha(1,4),2),3) - d(d(\alpha(1,4),3),2) + \underline{\alpha(1,e(\beta(2,3),4))} - \alpha(d(d(1,2),3),4)\\ 
 &  + \alpha(d(d(1,3),2),4) \stackrel{ Mor4 }{=}  \\ 
&\underline{d(d(\alpha(1,4),2),3)} - \underline{d(d(\alpha(1,4),3),2)} - \alpha(d(d(1,2),3),4) + \alpha(d(d(1,3),2),4)\\ 
 &  + \alpha(1,e(2,e(3,4))) - \alpha(1,d(e(2,4),3)) \stackrel{ Leib2 }{=}  \\ 
& - \alpha(d(d(1,2),3),4) + \alpha(d(d(1,3),2),4) + \alpha(1,e(2,e(3,4))) - \alpha(1,d(e(2,4),3))\\ 
 &  + \underline{d(\alpha(d(1,2),4),3)} + d(\alpha(1,e(2,4)),3) - \underline{d(\alpha(d(1,3),4),2)} - \underline{d(\alpha(1,e(3,4)),2)} \stackrel{ Leib2 }{=}  \\ 
& - \alpha(1,d(e(2,4),3)) + \underline{d(\alpha(1,e(2,4)),3)} - \alpha(d(1,3),e(2,4)) \stackrel{ Leib3 }{=}  0\\ 
\end{align*} 
 
{Reducing S-polynomial for relations SMod1 and SMod2} associated with the small common multiple $ n(m(1,m(3,4)),2)  
$: \begin{align*}&n(m(\alpha(1,3),4),2) - \underline{m(\alpha(1,2),m(3,4))} \stackrel{ SMod1 }{=}  \\ 
&\underline{n(m(\alpha(1,3),4),2)} - m(\alpha(\alpha(1,2),3),4) \stackrel{ SMod2 }{=}  \\ 
& - m(\alpha(\alpha(1,2),3),4) + \underline{m(\alpha(\alpha(1,3),2),4)} \stackrel{ Com1 }{=}  0\\ 
\end{align*} 
 
{Reducing S-polynomial for relations Com1 and Com2} associated with the small common multiple $ \alpha(\alpha(1,3),\alpha(2,4))  
$: \begin{align*}&\alpha(\alpha(1,\alpha(2,4)),3) - \underline{\alpha(\alpha(\alpha(1,3),2),4)} \stackrel{ Com1 }{=}  \\ 
&\underline{\alpha(\alpha(1,\alpha(2,4)),3)} - \alpha(\alpha(\alpha(1,2),3),4) \stackrel{ Com2 }{=}  \\ 
& - \alpha(\alpha(\alpha(1,2),3),4) + \underline{\alpha(\alpha(\alpha(1,2),4),3)} \stackrel{ Com1 }{=}  0\\ 
\end{align*} 
 
{Reducing S-polynomial for relations Com1 and Leib1} associated with the small common multiple $e(1,\alpha(\alpha(2,4),3)) 
$: \begin{align*}&e(1,\alpha(\alpha(2,3),4)) - \alpha(e(1,\alpha(2,4)),3) - \underline{\alpha(e(1,3),\alpha(2,4))} \stackrel{ Com2 }{=}  \\ 
&\underline{e(1,\alpha(\alpha(2,3),4))} - \underline{\alpha(e(1,\alpha(2,4)),3)} - \alpha(\alpha(e(1,3),2),4) \stackrel{ Leib1 }{=}  \\ 
& - \alpha(\alpha(e(1,3),2),4) + \underline{\alpha(e(1,\alpha(2,3)),4)} + \alpha(e(1,4),\alpha(2,3)) - \alpha(\alpha(e(1,2),4),3)\\ 
 &  - \alpha(\alpha(e(1,4),2),3) \stackrel{ Leib1 }{=}  \\ 
&\alpha(e(1,4),\alpha(2,3)) - \underline{\alpha(\alpha(e(1,2),4),3)} - \alpha(\alpha(e(1,4),2),3) + \alpha(\alpha(e(1,2),3),4) \stackrel{ Com1 }{=}  \\ 
&\underline{\alpha(e(1,4),\alpha(2,3))} - \alpha(\alpha(e(1,4),2),3) \stackrel{ Com2 }{=}  0\\ 
\end{align*}




\end{document}